\documentclass[reqno,12pt]{amsart}
\usepackage{amsmath,amssymb,amsfonts,amsthm, amscd,indentfirst}
\usepackage{amsmath,latexsym,amssymb,amsmath,
	amscd,amsthm,amsxtra}
 \usepackage{color}

\usepackage[hyphens]{url}
\usepackage[colorlinks=true,linkcolor=blue,citecolor=blue,urlcolor=blue,hypertexnames=false,linktocpage]{hyperref}
\usepackage{bookmark}
\usepackage{amsmath,thmtools,mathtools}
\mathtoolsset{showonlyrefs=true}

\usepackage[T1]{fontenc}
\usepackage{newtxtext,newtxmath}

\usepackage[msc-links]{amsrefs} 

\newcounter{mparcnt}

\usepackage{fancyhdr}
\usepackage{esint}
\usepackage{enumerate}

\usepackage{pictexwd,dcpic}
\usepackage{graphicx}
\usepackage{caption}
\usepackage{xcolor}
\usepackage{graphicx}
\usepackage{caption}

\usepackage{epsfig,here}
\usepackage{subfigure,here}

\newtheorem{theorem}{Theorem}[section]
\newtheorem{proposition}[theorem]{Proposition}
\newtheorem{definition}[theorem]{Definition}
\newtheorem{corollary}[theorem]{Corollary}
\newtheorem{remark}[theorem]{Remark}

\def\Om{\Omega}
\def\p{\partial}
\def\ep{\epsilon}
\def\de{\delta}

\def\S{{\Sigma}}
\def\<{\langle}
\def\>{\rangle}

\def\na{\nabla}

\providecommand{\abs}[1]{\lvert#1\rvert}
\providecommand{\Abs}[1]{\left\lvert#1\right\rvert}

\providecommand{\norm}[1]{\lVert#1\rVert}

\newcommand{\mbB}{\mathbb{B}}

\newcommand{\mbM}{\mathbb{M}}

\newcommand{\mbR}{\mathbb{R}}
\newcommand{\mbS}{\mathbb{S}}

\newcommand{\mcA}{\mathcal{A}}

\newcommand{\mcF}{\mathcal{F}}

\newcommand{\mcH}{\mathcal{H}}

\newcommand{\mcL}{\mathcal{L}}

\newcommand{\mcW}{\mathcal{W}}

\newcommand{\mfR}{\mathbf{R}}

\newcommand{\rd}{{\rm d}}

\newcommand{\ra}{\rightarrow}

\newcommand{\eq}[1]{\begin{equation}\begin{alignedat}{2} #1 \end{alignedat}\end{equation}}

\numberwithin{equation} {section}

\begin{document}
	
	\title[Monotonicity Formulas]{Monotonicity Formulas  for Capillary Surfaces}
\author[Wang]{Guofang Wang}
\address[G.W]{Mathematisches Institut\\
Universit\"at Freiburg\\
Ernst-Zermelo-Str.1\\
79104\\
\newline\indent Freiburg\\ Germany}
\email{guofang.wang@math.uni-freiburg.de}

\author[Xia]{Chao Xia}
\address[C.X]{School of Mathematical Sciences\\
Xiamen University\\
361005, Xiamen, P.R. China}
\email{chaoxia@xmu.edu.cn}

\author[Zhang]{Xuwen Zhang}
\address[X.Z]{Mathematisches Institut\\
	Universit\"at Freiburg\\
	Ernst-Zermelo-Str.1\\
	79104\\
	\newline\indent Freiburg\\ Germany}
\email{xuwen.zhang@math.uni-freiburg.de}
	
\begin{abstract}
In this paper, we establish monotonicity formulas for capillary surfaces in the half-space $\mbR^3_+$ and in the unit ball $\mbB^3$ and extend the result of Volkmann \cite{Volkmann16} for surfaces with free boundary.
As applications, we obtain Li-Yau-type inequalities  for the Willmore energy of capillary surfaces, and extend Fraser-Schoen's optimal area estimate for minimal free boundary surfaces in $\mbB^3$ \cite{FS11}
to the capillary setting,
which is different to another optimal area estimate proved by Brendle in \cite{Brendle23}.

\noindent{\bf Keywords:} monotonicity formula, Li-Yau inequality, Willmore energy, optimal area estimate, minimal surface, capillary surface \\

\noindent {\bf MSC 2020:} 53C42, 53A10, 49Q15 \\

\end{abstract}

\maketitle
\tableofcontents

\medskip

\section{Introduction}
For a $2$-dimensional immersed, open surface $\S\subset\mbR^{n+1}$, it is proved by Simon \cite{Simon93} that for $0<\sigma<\rho<\infty$, $a\in\mathbb{R}^{n+1}$,
\eq{\label{iden-Simon}
g_{a}(\rho)-g_a(\sigma)
=\frac{1}{\pi}\int_{\S\cap B_\rho(a)\setminus B_\sigma(a)}\Abs{\frac14\vec {\bf H}+\frac{(x-a)^\perp}{\abs{x-a}^2}}^2\rd\mcH^2,
}
where
\eq{
g_a(r)\coloneqq&
\frac{\mcH^2(\S\cap B_r(a))}{\pi r^2}+\frac{1}{16\pi}\int_{\S\cap B_r(a)}\abs{\vec {\bf H}}^2\rd\mcH^2
+\frac{1}{2\pi r^2}\int_{\S\cap B_r(a)}\vec {\bf H}\cdot(x-a)\rd\mcH^2.
}
This is known as Simon's monotonicity identity and is later generalized to hold for integral $2$-varifolds by Kuwert-Sch\"atzle \cite{KS04}.
As an interesting application, the monotonicity formula yields an alternative proof (by taking $\rho\ra\infty$ and $\sigma\ra0^+$) of the Li-Yau inequality (\cite{LY82}):
\eq{
\pi\Theta_{max}
\leq\frac{1}{16}\int_{\mbR^{n+1}}\abs{\vec{\bf H}}^2\rd\mu,
}
where $\Theta_{max}$ denotes the maximal density of an integral $2$-varifold $\mu$, whose generalized mean curvature in $\mbR^{n+1}$ is square integrable.
Note that for an immersion of a $2$-dimensional compact orientable closed smooth surface $F:\S\ra\mbR^{n+1}$,
let $\mu_g$ be the induced area measure of $\S$ with respect to the pull-back metric $g=F^\ast g_{\rm euc}$, then its image as integral $2$-varifold is given by
\eq{
\mu\coloneqq F(\mu_g)
=\left(x\mapsto\mcH^0(F^{-1}(x))\right)\mcH^2\llcorner F(\S).
}Hence
from the Li-Yau inequality, we easily see that the Willmore energy
\eq{
\mcW(\mu)
\coloneqq\frac14\int_{\mbR^{n+1}}\abs{\vec{\bf H}}^2\rd\mu
\geq4\pi,
}
and $F:\S\ra\mbR^{n+1}$ is an embedding
if
\eq{
\mcW(\mu)<8\pi.
}
We refer the interested readers to the monographs \cites{KS12,MN14} for an overview of the study of Willmore energy.

Recently, Volkmann \cite{Volkmann16} generalized this theory to free boundary surfaces in the unit ball $\mbB^{n+1}=\{x\in\mbR^{n+1}:\abs{x}<1\}$ (indeed, he proved it in the context of integral free boundary $2$-varifolds).
He obtained a monotonicity identity similar to \eqref{iden-Simon} (see Section \ref{Sec-4}) and consequently a Li-Yau-type inequality.
As an application, he established the Willmore energy for integral free boundary $2$-varifolds in $\mbB^{n+1}$, and proved that its lower bound is given exactly by $2\pi$.

\subsection{Main Result}
In this paper, we deal with the capillary counterpart of the above results.
We first prove a Simon-type monotonicity formula for capillary hypersurfaces in the half-space $$\mbR^3_+:=\{x\in\mbR^{3}:x_3<1\}$$ that are formulated in the weak sense, see Section \ref{Sec-2-1} for the precise definition.

\begin{theorem}[Simon-type monotonicity formula]\label{Thm-Simon-halfspace}
Given $\theta\in(0,\pi)$, let $V$ be a
rectifiable $2$-varifold supported on $\overline{\mbR^3_+}$ with its weight measure denoted by $\mu$ and let $W$  a rectifiable 2-varifold supported on $\p\mbR^3_+$ with its weight measure denoted by $\eta$, satisfying {a contact angle condition } as in Definition \ref{Defn-contact-pair} (and adopt the notations in Section \ref{Sec-2-1} below).
For any $a\in\mbR^3$, consider the functions $g_a(r),\hat g_a(r)$ defined by
\eq{
g_a(r)\coloneqq
\frac{\mu(B_r(a))}{\pi r^2}-\cos\theta\frac{\eta(B_r(a))}{\pi r^2}+\frac{1}{16\pi}\int_{B_r(a)}\abs{\vec {\bf H}}^2\rd\mu
+\frac{1}{2\pi r^2}\int_{B_r(a)}\vec{\bf H}\cdot(x-a)\rd\mu
,\\
\hat g_a(r)\coloneqq
\frac{\mu(\hat B_r(a))}{\pi r^2}-\cos\theta\frac{\eta(\hat B_r(a))}{\pi r^2}+\frac{1}{16\pi}\int_{\hat B_r(a)}\abs{\vec {\bf H}}^2\rd\mu
+\frac{1}{2\pi r^2}\int_{\hat B_r(a)}\vec {\bf H}\cdot(x-\tilde a)\rd\mu.
}
Then for any $0<\sigma<\rho<\infty$, we have
\eq{\label{identity-Simon-capillary}
&\left(g_a(\rho)+\hat g_a(\rho)\right)-\left(g_a(\sigma)+\hat g_a(\sigma)\right)\\
&=\frac{1}{\pi}\int_{B_\rho(a)\setminus B_\sigma(a)}\Abs{\frac14\vec {\bf H}+\frac{(x-a)^\perp}{\abs{x-a}^2}}^2\rd\mu
+\frac{1}{\pi}\int_{\hat B_\rho(a)\setminus \hat B_\sigma(a)}\Abs{\frac14\vec {\bf H}+\frac{(x-\tilde a)^\perp}{\abs{x-\tilde a}^2}}^2\rd\mu\\
&-2\cos\theta\int_{B_\rho(a)\setminus B_\sigma(a)}\frac{a_3^2}{\abs{x-a}^4}\rd\eta-2\cos\theta\int_{\hat B_\rho(a)\setminus \hat B_\sigma(a)}\frac{\tilde a_3^2}{\abs{x-\tilde a}^4}\rd\eta,
}
where $(x-a)^\perp\coloneqq
\nu_\S\cdot(x-a)\nu_\S$.
\end{theorem}

As an application, a Li-Yau-type inequality together with the characterization of the equality case is obtained for capillary immersions in the half-space.
For the precise definition of capillary immersions, see Section \ref{Sec-3-capillary-immersion}. We will adopt the notations that are used in Section \ref{Sec-3-capillary-immersion}.

The Willmore energy of capillary immersion $F$ is defined in the usual sense, that is,
\begin{definition}\label{Defn-Willmore-halfspace}
\normalfont
Given $\theta\in(0,\pi)$ and a $\theta$-capillary immersion $F:\S\ra\overline{\mbR^3_+}$.
The Willmore energy of the capillary immersion $F$ is defined as
\eq{
\mcW(F)
\coloneqq\frac14\int\abs{\vec {\bf H}}^2\rd\mu.
}
\end{definition}

 {$\mcW(F)$ is a conformal invariant with respect to conformal diffeomorphisms.} Detailed discussions regarding the definition of the Willmore functional for capillary surfaces are provided in Section \ref{Sec-3-1}.

\begin{theorem}\label{Thm-Li-Yau-Ineq0}
Given $\theta\in(0,\pi)$, let $F:\S\ra\overline{\mbR^3}_+$ be a compact $\theta$-capillary immersion, satisfying \eqref{condi-winding-number}.
Then  we have 
\begin{enumerate}
    \item 
the Li-Yau-type inequality
\eq{\label{eq_ly}
\mcW(F)
\geq4(1-\cos\theta) \pi\Theta^2(\mu,x_0)
}
holds for every $x_0\in F(\p\S)$, where $\Theta^2(\mu,x)$ denotes the density of $\mu$ at $x_0$.
\item 
the sharp estimate on Willmore energy:    
\eq{\label{eq_ly2}
    \mcW(F)\geq 2  (1-\cos\theta) \pi,
}
    and if \eq{
    \mcW(F)<4 (1-\cos\theta) \pi,
}
then $F:\p\S\subset\p\mbR^3_+$ must be an embedding.
\item 
For $\theta\in[\frac\pi2,\pi)$, if 
\eq{
W(F)<4\pi,
}
then $F:\S\ra\overline{\mbR^3}_+$ is an embedding.

\end{enumerate}Moreover, 
equality in \eqref{eq_ly2} holds if and only if $F(\S)$ is a $\theta$-spherical cap.
\end{theorem}

Inequality \eqref{eq_ly2} was proved using a flow method by the first and the second author in \cite{WWX23}*{Corollary 1.3} for convex capillary surfaces in $\overline{\mbR^3_+}$, while the contact angle $\theta$ is restricted within the range $(0,\frac\pi2]$.
Therefore Theorem 
\ref{Thm-Li-Yau-Ineq0} complements the full range of contact angle and moreover removes the assumption of the convexity.

Due to the conformal invariance of the Willmore functional, Theorem \ref{Thm-Li-Yau-Ineq0}
implies the same inequality as \eqref{eq_ly2} for capillary surfaces in the unit ball, which interestingly implies a new optimal area estimate for minimal capillary surfaces in $\mbB^3$.
\begin{theorem}
\label{coro3.5}
Given $\theta\in(0,\pi)$, let $F:\S\ra\overline{\mbB^3}$ be a $\theta$-capillary minimal immersion, satisfying \eqref{condi-winding-number}.
Then there holds
\eq{\label{eq_coro3.5}
2\abs{\S}-\cos\theta\abs{T}
\geq2(1-\cos\theta)\pi.
}
Equality holds if and only if $F(\S)$ is  a totally geodesic disk in $\overline{\mbB^3}$. Here $|T|$ is the area of the ``wetting'' part defined in Section \ref{Sec-3-capillary-immersion}.
\end{theorem}

As a  special case of the isoperimetric inequality for minimal submanifolds proved by Brendle \cite{Brendle21}, any capillary minimal surface $\S$ in the unit ball satisfies the following optimal estimate
\eq{\label{ineq-Brendle23-Thm5.5}
\abs{\S}
\geq\sin^2\theta\pi
=\sin^2\theta\abs{\mbB^2}
}
with equality holding if and only if $\S$ is a flat disk.
In fact, the higher dimensional counterparts also hold, see \cite{Brendle23}*{Theorem 5.5}.
Restricting $\theta=\frac\pi2$, i.e., in the free boundary case, this area lower bound was obtained  in \cite{FS11} and \cite{Brendle12}.
\eqref{eq_coro3.5} now provides a new optimal area estimate for minimal capillary surfaces in the unit ball. It seems that both area estimates do not imply each other. 
In view of \cite{Brendle23}, it would be interesting to seek for a higher dimensional generalization of \eqref{eq_coro3.5}.

In the last part of the paper, Section \ref{Sec-4}, we establish a monotonicity formula \eqref{iden-Simon-capil-x0} for capillary
surfaces in $\mbB^3$. 
This monotonicity formula is  a generalization of Volkmann's theorem  for surfaces with  free boundary in the unit ball \cite{Volkmann16}*{Theorem 3.1}.
Due to the extra 2-varifold $W$ supported on the unit sphere (as in Theorem \ref{Thm-Simon-halfspace}), our monotonicity formula \eqref{iden-Simon-capil-x0} becomes more involved  than his.
We refer the readers to Theorem \ref{Thm-in-ball} below.
Theorem \ref{coro3.5} can also be proved directly by 
this monotonicity formula.


We end the introduction with a short summary of  the fruitful results of the  study of 
capillary surfaces. The existence, regularity, and geometric properties of  capillary surfaces have attracted more and more attention from 
differential geometers and geometric analysts, see the nice book of Finn \cite{Finn86} for a through introduction.
Here we mention some recent progress on this topic. 
In $\mbR^3$, optimal boundary regularity for solutions to the capillarity problems (volume-constrained local minimizers of the free energy functional) was obtained by Taylor \cite{Taylor77}, while the partial regularity was obtained recently by De Philippis-Maggi \cites{DePM15,DePM17} in $\mbR^n (n\geq4)$, even in the anisotropic setting.
Quite recently, using the Min-Max method, the existence of capillary minimal or CMC hypersurfaces in compact $3$-manifolds with boundary has been shown independently by De Masi-De Philippis \cite{MP21} and Li-Zhou-Zhu \cite{LZZ21}.
In terms of the classical differential geometry, Hong-Saturnino \cite{HS23} carried out curvature and index estimates for compact and non-compact capillary surfaces.
In \cite{JWXZ22}, we proved a Heintze-Karcher-type inequality and give the characterization of smooth CMC capillary hypersurfaces in the half-space or in a wedge, namely, the Alexandrov-type theorem.
The anisotropic counterpart was tackled in \cites{JWXZ23,JWXZ23b} as well.
See also \cite{XZ23} for a non-smooth generalization.

\subsection{Organization of the Paper}
In Section \ref{Sec-2}
we first introduce the capillary surface in the half-space in the setting of varifolds and then
provide the proof of Theorem \ref{Thm-Simon-halfspace}, 
In Section \ref{Sec-3-1} we use the monotonicity formula to estimate the Willmore energy for capillary surfaces in the half-space and prove
the Li-Yau inequality \eqref{eq_ly} and Theorem \ref{Thm-Li-Yau-Ineq0}. The Willmore energy for capillary surfaces in  the unit ball is discussed in Section \ref{sec-3-2}, together with one proof of Theorem \ref{coro3.5}.
Finally, in Section \ref{Sec-4}, we prove the Simon-type monotonicity formulas in the unit ball, and then provide an alternative proof of Theorem \ref{coro3.5}.

\section{Monotonicity Formula in the Half-Space}\label{Sec-2}

\subsection{Set-ups}\label{Sec-2-1}
We will be working on the Euclidean space $\mbR^3$, with  the Euclidean metric  denoted by $g_{\rm euc}$ and the corresponding Levi-Civita connection denoted by $\na$. $\mbR^3_+=\{x:x_3>0\}$ is the open upper half-space and $E_3\coloneqq(0,0,1)$.

Let $\xi:\mbR^3\ra\p\mbR^3_+$ denote the unique point projection onto the hyperplane $\p\mbR^{3}_+$.
It is easy to see that for any $x\in\mbR^3$, $\xi(x)=x-x_3E_3$, and hence $\xi$ is a smooth map.
Define
\eq{
\tilde x\coloneqq2\xi(x)-x
=x-2x_3E_3
}
to be the reflection of $x$ across $\p\mbR^3_+$.
Given a point $a\in\mbR^3$, define
\eq{
r=\abs{x-a},\quad
\tilde r=\abs{\tilde x-a},
}
it is clear  that $\tilde r=\abs{\tilde x-a}
=\abs{x-\tilde a}$.

We refer to \cite{Sim83} for the background materials in Geometric Measure Theory and consider the following weak formulation of capillary surfaces. 
\begin{definition}\label{Defn-contact-pair}
\normalfont
Given $\theta\in(0,\pi)$, let $V$ be a rectifiable $2$-varifold supported on $\overline{\mbR^3_+}$ with its weight measure denoted by $\mu$ and let $W$  a rectifiable 2-varifold supported on $\p\mbR^3_+$ with its weight measure denoted by $\eta$.
$(V,W)$ is said to satisfy \textit{a contact angle condition }$\theta$ if there exists a $\mu$-measurable vector field $\vec {\bf H}\in\mcL^1(\mbR^3,\mu)$ with $\vec {\bf H}(x)\in T_x\p\mbR^3_+$ for $\mu$-a.e. $x\in\p\mbR^3_+$,
such that for every $X\in C_c^1(\mbR^3;\mbR^3)$ with $X$
tangent to $\p\mbR^3_+$, 
\eq{\label{eq-1st-variation-H-3}
\int_{\mbR^3}{\rm div}_\S X\rd\mu-\cos\theta\int_{\p\mbR^3_+}{\rm div}_{\p\mbR^3_+}X\rd\eta
=-\int_{\mbR^3}\vec {\bf H}\cdot X\rd\mu,
}
where $\S\coloneqq{\rm spt}\,\mu$ is countably $2$-rectifiable, for simplicity we also set $T\coloneqq{\rm spt}\,\eta$.
In particular, if $X$ vanishes along $\p\mbR^3_+$, there holds that
\eq{
\int_{\mbR^3}{\rm div}_\S X\rd\mu=-\int_{\mbR^3}\vec {\bf H}\cdot X\rd\mu.
}
\end{definition}

This definition was introduced in \cite{MP21}, which follows from the one initiated in \cite{KT17},
with the boundary part weakened to be just a $2$-varifold.

Notice that the first variation formula \eqref{eq-1st-variation-H-3} is valid if $X$ is merely a Lipschitz vector field.

In this section, we consider those $V,W$ that are integral $2$-varifolds with  $\mu(\p\mbR^3_+)=0$
and $\vec {\bf H}\in\mcL^2(\mbR^3,\mu)$.
It follows that $\vec{\bf H}\perp\S$ for $\mu$-a.e. $x\in\S$ thanks to the well known Brakke's perpendicular theorem (\cite{Brakke78}*{Sect. 5.8}), so that for any vector $v\in\mbR^3$, one has
\eq{\label{eq-Volkmann16-(8)}
2\Abs{\frac14\vec {\bf H}+v^\perp}^2
=\frac18\abs{\vec {\bf H}}^2+2\abs{v^\perp}^2+\vec {\bf H}\cdot v,
}
where $v^\perp$ denotes the normal part of $v$ with respect to the approximate tangent space $T_x\S$.

\subsection{A Monotonicity Formula}
A general monotonicity formula for the pairs of varifolds satisfying a contact condition has  been established in \cite{MP21} (see also \cite{KT17}) with none sharp constants, which generalizes a previous result in \cite{GJ86} for rectifiable free boundary varifolds.
We are interested in the Simon-type monotonicity formula \eqref{iden-Simon} for the capillary case with optimal constants.
\begin{proof}[Proof of Theorem \ref{Thm-Simon-halfspace}]
We follow closely the reflection idea in \cite{GJ86}. See also \cite{KT17}.
Define $l(s)\coloneqq\left((\frac{1}{s_\sigma})^2-\frac{1}{\rho^2}\right)_+$,
where $s_\sigma\coloneqq\max\{s,\sigma\}$.
$l(s)$ is then a Lipschitz cut-off function and it is easy to see that
\eq{
l(s)
=
\begin{cases}
    \frac{1}{\sigma^2}-\frac{1}{\rho^2},\quad&0<s\leq\sigma,\\
    \frac{1}{s^2}-\frac{1}{\rho^2},\quad&\sigma<s\leq\rho,\\
    0,\quad&\rho<s.
\end{cases}
}
We wish to find a suitable vector field to test \eqref{eq-1st-variation-H-3}. The construction is inspired by \cite{GJ86} and is nowadays quiet standard.
Precisely,
we define
\eq{
&X_1(x)=\left((\frac{1}{\abs{x-a}_\sigma})^2-\frac{1}{\rho^2}\right)_+(x-a),\\
&X_2(x)=\left((\frac{1}{\abs{x-\tilde a}_\sigma})^2-\frac{1}{\rho^2}\right)_+(x-\tilde a),\\
&X(x)\coloneqq X_1(x)+X_2(x).
}
For $x\in\p\mbR^3_+$, since $x=\tilde x$ we simply have
\eq{
\abs{x-a}_\sigma
=\abs{\tilde x-a}_\sigma
=\abs{x-\tilde a}_\sigma,
}
and hence
\eq{
X(x)=\left((\frac{1}{\abs{x-a}_\sigma})^2-\frac{1}{\rho^2}\right)_+(2x-(a+\tilde a))\in\p\mbR^3_+.
}

Let $\hat B_r(a)=\{x:\abs{x-\tilde a}<r\}$
and consider the partitions of $\mbR^3$:
\eq{
\mcF_1\coloneqq\{B_\sigma(a),B_\rho(a)\setminus B_\sigma(a),\mbR^3\setminus B_\rho(a)\},\\
\mcF_2\coloneqq\{\hat B_\sigma(a),\hat B_\rho(a)\setminus \hat B_\sigma(a),\mbR^3\setminus \hat B_\rho(a)\}.
}
We wish to test \eqref{eq-1st-variation-H-3} by $X(x)$. To this end, we compute
\eq{
\int_A{\rm div}_\S X_i\rd\mu-\cos\theta\int_A{\rm div}_{\p\mbR^3_+}X_i\rd\eta,\text{ and }\int_A\vec {\bf H}\cdot X_i\rd\mu
}
for all sets $A\in\mcF_i$, $i=1,2$, separately.

\noindent{\bf For $X_2$:} on $0\leq\abs{x-\tilde a}\leq\sigma$, direct computation shows that
\eq{
\na X_2(x)=(\frac{1}{\sigma^2}-\frac{1}{\rho^2}){\rm Id},
}
and hence
\eq{
\int_{\hat B_\sigma(a)}{\rm div}_\S X_2\rd\mu-\cos\theta\int_{\hat B_\sigma(a)}{\rm div}_{\p\mbR^3_+}X_2\rd\eta
&=(\frac{2}{\sigma^2}-\frac{2}{\rho^2})\left(\mu(\hat B_\sigma(a))-\cos\theta\eta(\hat B_\sigma(a))\right),\\
\int_{\hat B_\sigma(a)}\vec {\bf H}\cdot X_2\rd\mu
&=(\frac{1}{\sigma^2}-\frac{1}{\rho^2})\int_{\hat B_\sigma(a)}\vec {\bf H}\cdot(x-\tilde a)\rd\mu.
}
On $\sigma<\abs{x-\tilde a}\leq\rho$, we have $X_2(x)=\left(\frac{1}{\abs{x-\tilde a}^2}-\frac{1}{\rho^2}\right)(x-\tilde a)$. By  noticing that
\eq{
\na(\frac{1}{\abs{x-\tilde a}^2})
=-2\frac{1}{\abs{x-\tilde a}^4}(x-\tilde a),
}
we have
\eq{
\na X_2(x)
&=\left(\frac{1}{\abs{x-\tilde a}^2}-\frac{1}{\rho^2}\right){\rm Id}-2\frac{1}{\abs{x-\tilde a}^4}(x-\tilde a)\otimes(x-\tilde a),\\
{\rm div}X_2(x)
&=3\left(\frac{1}{\abs{x-\tilde a}^2}-\frac{1}{\rho^2}\right)-\frac{2}{\abs{x-\tilde a}^2}
=\frac{1}{\abs{x-\tilde a}^2}-\frac{3}{\rho^2},\\
{\rm div}_\S X_2(x)
&={\rm div}X_2(x)-\na X_2[\nu_\S]\cdot\nu_\S=-\frac{2}{\rho^2}+2\Abs{\frac{(x-\tilde a)^\perp}{\abs{x-\tilde a}^2}}^2,\\
{\rm div}_{\p\mbR^3_+}X_2(x)&={\rm div}X_2(x)-\na X_2[E_3]\cdot E_3
=-\frac{2}{\rho^2}+\frac{2\tilde a_3^2}{\abs{x-\tilde a}^4}.
}
It follows that
\eq{&
\int_{\hat B_\rho(a)\setminus\hat B_\sigma(a)}{\rm div}_{\S}X_2\rd\mu-\cos\theta\int_{\hat B_\rho(a)\setminus\hat B_\sigma(a)}{\rm div}_{\p\mbR^3_+}X_2\rd\eta \\
=&-\frac{2}{\rho^2}\left(\mu(\hat B_{\rho}(a)\setminus \hat B_\sigma(a))-\cos\theta\eta(\hat B_{\rho}(a)\setminus \hat B_\sigma(a))\right)
+2\int_{\hat B_\rho(a)\setminus\hat B_\sigma(a)}\Abs{\frac{(x-\tilde a)^\perp}{\abs{x-\tilde a}^2}}^2\rd\mu\\
&-2\cos\theta\int_{\hat B_\rho(a)\setminus\hat B_\sigma(a)}\frac{\tilde a_3^2}{\abs{x-\tilde a}^4}\rd\eta,\\&
\int_{\hat B_\rho(a)\setminus\hat B_\sigma(a)}\vec {\bf H}\cdot X_2\rd\mu
=-\frac{1}{\rho^2}\int_{\hat B_\rho(a)\setminus\hat B_\sigma(a)}\vec {\bf H}\cdot(x-\tilde a)\rd\mu
+\int_{\hat B_\rho(a)\setminus\hat B_\sigma(a)}\vec {\bf H}\cdot\frac{x-\tilde a}{\abs{x-\tilde a}^2}\rd\mu.
}
Combining these computations, we obtain
\eq{ &
\int_{\mbR^3}{\rm div}_\S X_2\rd\mu-\cos\theta\int_{\mbR^3}{\rm div}_{\p\mbR^3_+}X_2\rd\eta\\
=&\frac{2}{\sigma^2}\left(\mu(\hat B_\sigma(a))-\cos\theta\eta(\hat B_\sigma(a))\right)-\frac{2}{\rho^2}\left(\mu(\hat B_\rho(a))-\cos\theta\eta(\hat B_\rho(a))\right)\\
&+2\int_{\hat B_\rho(a)\setminus\hat B_\sigma(a)}\Abs{\frac{(x-\tilde a)^\perp}{\abs{x-\tilde a}^2}}^2\rd\mu-2\cos\theta\int_{\hat B_\rho(a)\setminus\hat B_\sigma(a)}\frac{a_3^2}{\abs{x-\tilde a}^4}\rd\eta,
}
and
\eq{
\int_{\mbR^3}\vec {\bf H}\cdot X_2\rd\mu
=&\frac{1}{\sigma^2}\int_{\hat B_\sigma(a)}\vec {\bf H}\cdot(x-\tilde a)\rd\mu
-\frac{1}{\rho^2}\int_{\hat B_\rho(a)}\vec {\bf H}\cdot(x-\tilde a)\rd\mu\\
&+\int_{\hat B_\rho(a)\setminus\hat B_\sigma(a)}\vec {\bf H}\cdot\frac{x-\tilde a}{\abs{x-\tilde a}^2}\rd\mu.
}
Similar computations hold for $X_1$ and we obtain
\eq{&
\int_{\mbR^3}{\rm div}_\S X_1\rd\mu-\cos\theta\int_{\mbR^3}{\rm div}_{\p\mbR^3_+}X_1\rd\eta\\
=&\frac{2}{\sigma^2}\left(\mu(B_\sigma(a))-\cos\theta\eta( B_\sigma(a))\right)-\frac{2}{\rho^2}\left(\mu( B_\rho(a))-\cos\theta\eta( B_\rho(a))\right)\\
&+2\int_{ B_\rho(a)\setminus B_\sigma(a)}\Abs{\frac{(x- a)^\perp}{\abs{x- a}^2}}^2\rd\mu-2\cos\theta\int_{ B_\rho(a)\setminus B_\sigma(a)}\frac{a_3^2}{\abs{x- a}^4}\rd\eta,
}
and
\eq{
\int_{\mbR^3}\vec {\bf H}\cdot X_1\rd\mu
=&\frac{1}{\sigma^2}\int_{B_\sigma(a)}\vec {\bf H}\cdot(x- a)\rd\mu
-\frac{1}{\rho^2}\int_{B_\rho(a)}\vec {\bf H}\cdot(x-a)\rd\mu\\
&+\int_{ B_\rho(a)\setminus B_\sigma(a)}\vec {\bf H}\cdot\frac{x-a}{\abs{x-a}^2}\rd\mu.
}

On the other hand, thanks to \eqref{eq-Volkmann16-(8)}, we have
\eq{
\int_{ B_\rho(a)\setminus B_\sigma(a)}\vec {\bf H}\cdot\frac{x-a}{\abs{x-a}^2}\rd\mu
=2\int_{ B_\rho(a)\setminus B_\sigma(a)}\Abs{\frac14\vec {\bf H}+\frac{(x-a)^\perp}{\abs{x-a}^2}}^2\rd\mu\\
-\frac18\int_{ B_\rho(a)\setminus B_\sigma(a)}\abs{\vec {\bf H}}^2\rd\mu-2\int_{ B_\rho(a)\setminus B_\sigma(a)}\Abs{\frac{(x-a)^\perp}{\abs{x-a}^2}}^2\rd\mu,\\
\int_{ \hat B_\rho(a)\setminus \hat B_\sigma(a)}\vec {\bf H}\cdot\frac{x-\tilde a}{\abs{x-\tilde a}^2}\rd\mu
=2\int_{ \hat B_\rho(a)\setminus \hat B_\sigma(a)}\Abs{\frac14\vec {\bf H}+\frac{(x-\tilde a)^\perp}{\abs{x-\tilde a}^2}}^2\rd\mu\\
-\frac18\int_{ \hat B_\rho(a)\setminus \hat B_\sigma(a)}\abs{\vec {\bf H}}^2\rd\mu-2\int_{ \hat B_\rho(a)\setminus \hat B_\sigma(a)}\Abs{\frac{(x-\tilde a)^\perp}{\abs{x-\tilde a}^2}}^2\rd\mu.
}
Putting these facts into the first variation formula \eqref{eq-1st-variation-H-3} and recalling the definition of $g_a(r),\hat g_a(r)$, we deduce \eqref{identity-Simon-capillary}.
\end{proof}

\begin{remark}
\normalfont We remark that in Theorem \ref{Thm-Simon-halfspace}:
\begin{enumerate} 
\item $B_r (a)\cap \p\mbR^3_+=\hat B_r(a)\cap \p\mbR^3_+$, and hence $\eta (B_r(a))=\eta(\hat B_r(a))$.
\item $a_3^2=\tilde a_3^2$ and hence the corresponding terms involving them are the same.
    \end{enumerate}
\end{remark}

\begin{proposition}\label{Prop-uppersemi-continuous}
Under the assumptions of Theorem \ref{Thm-Simon-halfspace}, the following statements hold:
\begin{enumerate}
\item Given $\theta\in[\frac{\pi}{2},\pi)$. For every $a\in\mbR^3$, the tilde-density
\eq{
\tilde\Theta^2(\mu-\cos\theta\eta,a)
\coloneqq\lim_{r\searrow0}\left(\frac{(\mu-\cos\theta\eta)(B_r(a))}{\pi r^2}+\frac{(\mu-\cos\theta\eta)(\hat B_r(a))}{\pi r^2}\right)
}
exists.
Moreover, the function $x\mapsto\tilde\Theta(\mu-\cos\theta\eta,x)$ is upper semi-continuous in $\mbR^3$;
\item Given $\theta\in(0,\pi)$.
For every $a\in\p\mbR^3$, the tilde-density $\tilde\Theta^2(\mu-\cos\theta\eta,a)$ exists, and the function $x\mapsto\tilde\Theta(\mu-\cos\theta\eta,x)$ is upper semi-continuous in $\p\mbR^3_+$.
\end{enumerate}

\end{proposition}

\begin{proof}
Define
\eq{
R(r)
&\coloneqq\frac{1}{2\pi r^2}\int_{ B_r(a)}\vec {\bf H}\cdot(x- a)\rd\mu+\frac{1}{2\pi r^2}\int_{\hat B_r(a)}\vec {\bf H}\cdot(x-\tilde a)\rd\mu,
}
and
\eq{
G_\theta(r)
\coloneqq &\frac{(\mu-\cos\theta\eta)(B_r(a))}{\pi r^2}+\frac{(\mu-\cos\theta\eta)(\hat B_r(a))}{\pi r^2}\\
&+\frac{1}{16\pi}\int_{B_r(a)}\abs{\vec {\bf H}}^2\rd\mu+\frac{1}{16\pi}\int_{\hat B_r(a)}\abs{\vec {\bf H}}^2\rd\mu+R(r),
}
then from \eqref{identity-Simon-capillary} and thanks to $\theta\in[\pi/2,\pi)$, we know that $G_\theta(r)$ is monotonically nondecreasing, so that
$
\lim_{r\ra0^+}G_\theta(r)
$
exists.


For $R(r)$, we estimate with H\"older inequality:
\eq{\label{ineq-KS04-(A.5)}
\abs{R(r)}
\leq &\left(\frac{\mu(B_r(a))}{\pi r^2}\right)^{\frac12}\left(\frac{1}{4\pi}\int_{B_r(a)}\abs{\vec {\bf H}}^2\rd\mu\right)^{\frac12}\\
&+\left(\frac{\mu(\hat B_r(a))}{\pi r^2}\right)^{\frac12}\left(\frac{1}{4\pi}\int_{\hat B_r(a)}\abs{\vec {\bf H}}^2\rd\mu\right)^{\frac12}.
}
Moreover, for $\frac14<\ep<\frac12$, Young's inequality gives
\eq{
\abs{R(r)}
\leq &\ep \frac{\mu(B_r(a))}{\pi r^2}+\frac{1}{16\pi\ep}\int_{B_r(a)}\abs{\vec {\bf H}}^2\rd\mu\\
&+\ep\frac{\mu(\hat B_r(a))}{\pi r^2}+\frac{1}{16\pi\ep}\int_{\hat B_r(a)}\abs{\vec {\bf H}}^2\rd\mu.
}

By virtue of the monotonicity of $G_\theta(r)$, we obtain
\eq{ &
\frac{(\mu-\cos\theta\eta)(B_\sigma(a))}{\pi \sigma^2}+\frac{(\mu-\cos\theta\eta)(\hat B_\sigma(a))}{\pi\sigma^2}\\
\leq &\frac{(\mu-\cos\theta\eta)(B_\rho(a))}{\pi \rho^2}+\frac{(\mu-\cos\theta\eta)(\hat B_\rho(a))}{\pi\rho^2}\\
&+\frac{1}{16\pi}\int_{B_\rho(a)}\abs{\vec {\bf H}}^2\rd\mu+\frac{1}{16\pi}\int_{\hat B_\rho(a)}\abs{\vec {\bf H}}^2\rd\mu+\abs{R(\rho)}+\abs{R(\sigma)}\\
\leq&(1+\ep)\left(\frac{(\mu-\cos\theta\eta)(B_\rho(a))}{\pi \rho^2}+\frac{(\mu-\cos\theta\eta)(\hat B_\rho(a))}{\pi \rho^2}\right)\\
&+\frac{1+2\ep^{-1}}{16\pi}(\int_{B_\rho(a)}\abs{\vec {\bf H}}^2\rd\mu+\int_{\hat B_\rho(a)}\abs{\vec {\bf H}}^2\rd\mu)\\
&+\ep\left(\frac{(\mu-\cos\theta\eta)(B_\sigma(a))}{\pi \sigma^2}+\frac{(\mu-\cos\theta\eta)(\hat B_\sigma(a))}{\pi\sigma^2}\right).
}
In particular, since $\frac14<\ep<\frac12$, this yields for $0<r<R<\infty$
\eq{\label{eq-KS04-(A.6)}
&\frac{(\mu-\cos\theta\eta)(B_r(a))}{\pi r^2}+\frac{(\mu-\cos\theta\eta)(\hat B_r(a))}{\pi r^2}\\
\leq&3\left(\frac{(\mu-\cos\theta\eta)(B_R(a))}{\pi R^2}+\frac{(\mu-\cos\theta\eta)(\hat B_R(a))}{\pi R^2}\right)+\frac{9}{8\pi}\int_{\mbR^3}\abs{\vec {\bf H}}^2\rd\mu
<\infty,
}
which, together with \eqref{ineq-KS04-(A.5)}, gives
\eq{
\lim_{r\ra0^+}R(r)=0,
}
implying that the tilde-density $\tilde\Theta(\mu-\cos\theta\eta,a)$ exists, and also
\eq{
\tilde\Theta(\mu-\cos\theta\eta,a)
\leq \left(\frac{(\mu-\cos\theta\eta)(B_\rho(a))}{\pi \rho^2}+\frac{(\mu-\cos\theta\eta)(\hat B_\rho(a))}{\pi \rho^2}\right)\\
+\frac{1}{16\pi}(\int_{B_\rho(a)}\abs{\vec {\bf H}}^2\rd\mu+\int_{\hat B_\rho(a)}\abs{\vec {\bf H}}^2\rd\mu)+R(\rho).
}
Thus for a fixed $R>0$ and a sequence of points in $\mbR^3$ such that $x_j\ra a$, we have: for $0<\rho<R$, it holds that
\eq{ &
\frac{(\mu-\cos\theta\eta)(\overline{B_\rho(a)})}{\pi\rho^2}+\frac{(\mu-\cos\theta\eta)(\overline{\hat B_\rho(a)})}{\pi\rho^2}\\
&\geq\limsup_{j\ra\infty}\left(\frac{(\mu-\cos\theta\eta)(\overline{B_\rho(x_j)})}{\pi\rho^2}+\frac{(\mu-\cos\theta\eta)(\overline{\hat B_\rho(x_j)})}{\pi\rho^2}\right)\\
&\geq\limsup_{j\ra\infty}\left(\tilde\Theta(\mu-\cos\theta\eta,x_j)-\frac{1}{16\pi}(\int_{B_\rho(x_j)}\abs{\vec {\bf H}}^2\rd\mu+\int_{\hat B_\rho(x_j)}\abs{\vec {\bf H}}^2\rd\mu)-R(\rho)\right)\\&
\overset{\eqref{eq-KS04-(A.6)}}{\geq}\limsup_{j\ra\infty}\tilde\Theta(\mu-\cos\theta\eta,x_j)\\
&\quad\, \,-C\left(\frac{(\mu-\cos\theta\eta)(B_R(a))+(\mu-\cos\theta\eta)(\hat B_R(a))}{\pi R^2}+\mcW(\mu)\right)^\frac12\norm{\vec {\bf H}}_{L^2(B_{2\rho}(a))}.
}
Letting $\rho\ra0^+$, this gives
\eq{
\tilde\Theta^2(\mu-\cos\theta\eta,a)
\geq\limsup_{j\ra\infty}\tilde\Theta(\mu-\cos\theta\eta,x_j),
}
which completes the proof of (1).

To prove (2), notice that for any $a\in\p\mbR^3_+$, $a_3=\tilde a_3=0$, and hence as $\theta\in(0,\pi)$, the function $G_\theta(r)$ is also monotonically nondecreasing for any such $a$.
Following the proof of (1), we conclude (2).
\end{proof}
\begin{remark}
\normalfont
Suppose that $\mu,\eta$ are compactly supported, we may thus use \eqref{ineq-KS04-(A.5)} to find that: for $\theta\in[\frac{\pi}{2},\pi)$, $a\in\mbR^3$,
\eq{
\lim_{r\ra\infty}\abs{R(r)}=0,
}
and obtain
\eq{
\lim_{r\ra\infty}\left(g_a(r)+\hat g_a(r)\right)
=\frac{1}{8\pi}\int_{\mbR^3}\abs{\vec {\bf H}}^2\rd\mu.
}
Thus by letting $\sigma\ra0^+$ and $\rho\ra\infty$ in \eqref{identity-Simon-capillary}, we obtain
\eq{\label{iden-monotone-Theta}
&\frac{1}{\pi}\int_{\mbR^3}\Abs{\frac14\vec {\bf H}+\frac{(x-a)^\perp}{\abs{x-a}^2}}^2\rd\mu
+\frac{1}{\pi}\int_{\mbR^3}\Abs{\frac14\vec {\bf H}+\frac{(x-\tilde a)^\perp}{\abs{x-\tilde a}^2}}^2\rd\mu\\
&-2\cos\theta\int_{\mbR^3}\frac{a_3^2}{\abs{x-a}^4}\rd\eta-2\cos\theta\int_{\mbR^3}\frac{\tilde a_3^2}{\abs{x-\tilde a}^4}\rd\eta\\
=&\frac{1}{8\pi}\int_{\mbR^3}\abs{\vec {\bf H}}^2\rd\mu-\tilde\Theta^2(\mu-\cos\theta\eta,a).
}
This is also true for any $a\in\mbR^3$ and any $\theta\in(0,\frac{\pi}{2})$, provided that $\tilde\Theta^2(\mu-\cos\theta\eta,a)$ exists.

Similarly, for $\theta\in(0,\pi)$ and $a\in\p\mbR^3_+$, we have
\eq{\label{eq2.6}
\frac{2}{\pi}\int_{\mbR^3}\Abs{\frac14\vec {\bf H}+\frac{(x-a)^\perp}{\abs{x-a}^2}}^2\rd\mu
=\frac{1}{8\pi}\int_{\mbR^3}\abs{\vec {\bf H}}^2\rd\mu-\tilde\Theta^2(\mu-\cos\theta \eta,a).
}
\end{remark}

\section{Willmore Energy and Li-Yau-Type Inequality for Capillary Immersions}
\label{Sec-3}
\subsection{Capillary Immersions}\label{Sec-3-capillary-immersion}

Given a  compact orientable smooth surface $\S$, with non-empty boundary $\p\S$.
Let $F:\S\ra\overline{\mbR^3_+}$ be an {\color{black}orientation preserving }proper immersion, that is, $F$ is smooth on the interior of $\S$ and $C^2$ up to the boundary, such that $F({\rm int}\S)\subset\mbR^3_+$ and $F(\p\S)\subset\p\mbR^3_+$. 
In this way, $F$ induces an immersion of $\p\S$ into $\p\mbR^3_+$.

{\color{black}
Fix a global unit normal field $\nu$ to $\S$ along $F$, which determines an orientation on $\S$ and an induced orientation on $\p\S$ given by a tangential vector field $\tau$ along $\p\S$.
Denote by $\mu$ the unit conormal to $\p\S$ in $\S$ so that $\{\tau,\nu, \mu\}$ is compatible with $\{E_1, E_2, E_3\}$ of $\mbR^3$. Let $\bar\nu$ be the unit normal to $\p\S$ in $\p\mbR^3_+$ so that $\{\bar\nu, -E_3\}$ compatible with $\{\nu,\mu\}$.
Given $\theta\in(0,\pi)$, we say that the proper immersion $F:\S\ra\overline{\mbR^3_+}$ is a {\it capillary immersion (or $\theta$-capillary immersion)} with contact angle $\theta$ if the angle determined by $\mu$ and $\bar\nu$ is constant and equals to $\theta$ along $\p\S$, that is, for any $x\in\p\S$,
\eq{\label{condi-capillary}
	\mu(x)=\sin\theta(-E_3)+\cos\theta\bar\nu(x).
}

Abuse of terminology, for the immersion $F: \p\S\to \p\mbR^3_+$, we use $T$ to denote the ``enclosed domain" by $F(\p\S)$, which means the topological boundary $\p T$ is given by $F(\p\S)$ and it has induced orientation by $\tau$. We use $\abs{T}$ to denote the oriented area of $T$, which is defined to be $$\abs{T}=\sum_i {\rm sgn}(T_i)|T_i|,$$
where each $T_i$ is a bounded domain such that $\p T_i$ is a simply closed curve and ${\rm sgn}(T_i)=+1$ if $\bar \nu$ points outward $T_i$, while
${\rm sgn}(T_i)=-1$ if $\bar \nu$ points inward $T_i$.} $|T|$ is the signed area of the so-called ``wetting part''.


{\color{black}
Now given  a $\theta$-capillary immersion $F:\S\ra\overline{\mbR^3_+}$ for $\theta\in(0,\pi)\setminus\{\frac\pi2\}$.
Let $g=F^\ast g_{\rm euc}$ be the pull-back metric and $\mu_g$ be the induced area measure on $\S$.
The induced varifold of $\S$ is given by $V_g=\mu_g\otimes\de_{T_p\S}$, where $T_p\S$ is the tangent space  and
its image as varifold is given by $V=F_\sharp(V_g)$, which is an integral $2$-varifold in $\mfR^3$ with weight measure $\mu
\coloneqq\left(x\mapsto\mcH^0(F^{-1}(x))\right)\mcH^2\llcorner F(\S)$, see \cite{Sim83}*{Section 15}.

Define the following measure:
\eq{
	\eta\coloneqq&{\rm wind}(x)\mcH^2,
}
where ${\rm wind}(x)$ is the winding number of $F(\p\S)$ about $x\in\p\mfR^3$. 
Throughout the paper, in terms of capillary immersion, we always assume that
\eq{\label{condi-winding-number}
{\rm wind}(x)\geq0\text{ for any }x\in\p\mbR^3_+.
}

With this assumption, $\eta$ is a positive Radon measure,
and we may relate to it naturally the $2$-varifold:
\eq{
	W=\eta\otimes\de_{T_x\p\mbR^3_+}.
}
Let $T={\rm spt}\eta\subset\p\mbR^3_+$.
By definition of winding number,  the mass $\mbM(W)$ of $W$ is exactly the oriented area of $T$ (In fact, ${\rm sgn}(T_i)$ will be always $+1$ thanks to \eqref{condi-winding-number}).
\begin{proposition}
    The pair $(V,W)$ above satisfies the contact angle condition $\theta$ with square integrable generalized mean curvature as in Definition \ref{Defn-contact-pair}.
\end{proposition}
\begin{proof}
In the case that $\S$ is an open surface and $F$ is a properly immersion, the fact that $V$ has square integrable generalized mean curvature in $\mfR^3$ has already been observed in \cite{KS12}*{Section 2.1}.
Here we prove the capillary version.

Let us consider any $X\in C_c^1(\mbR^3;\mbR^3)$ which is tangent to $\p\mbR^3_+$, and let $\{f_t:\mbR^3\ra\mbR^3\}_{\abs{t}<\ep}$ be the induced variation of $X$, that is, 
\eq{
	f_0={\rm id}, f_t(\p\mbR^3_+)\subset\p\mbR^3_+,
	\text{ and } \frac{\rd}{\rd t}\mid_{t=0}f_t=X,
}
which induces a family of immersions: $\{\psi_t\}_{\abs{t}<\ep}$, where $\psi_t=f_t\circ F:\S\ra\overline{\mbR^3_+}$.
$\Psi:(-\ep,\ep)\times\S\ra\overline{\mbR^3_+}$, given by $\Psi(t,p)=\psi_t(p)$, is called an admissible variation.
Let $\xi(p)=\frac{\p\Psi}{\p t}(0,p)$ for $p\in\S$, then $\xi(p)=X(F(p))$.

For this family of immersions, let $\mu_{g_t}$ be the induced area measure of $\S$ with respect to the pull-back metric $g_t=\psi^*_tg_{\rm euc}$, the area functional is then given by
\eq{
	\mcA:(-\ep,\ep)\ra\mbR,\quad 
	\mcA(t)=\int_\S\rd\mu_{g_t}.
}
The wetted area functional $\mcW(t):(-\ep,\ep)\ra\mbR$ is defined by
\eq{
	\mcW(t)
	=\int_{[0,t]\times\p\S}\Psi^\ast\Om,
}
where $\Om$ is the volume form of $\p\mbR^3_+$.
We define the free energy functional by
\eq{
	\mcF:(-\ep,\ep)\ra\mbR,\quad
	\mcF(t)=\mcA(t)-\cos\theta\mcW(t),
}
from the well-known first variation formula (see e.g., \cite{AS16}*{(2.1)}) and the capillary condition, we find (recall that $\xi(p)=X(F(p))$)
\eq{
	\mcF'(0)
	=&-\int_\S\mathbf{H}_\S(p)\left<\xi(p),\nu(p)\right>\rd\mu_g\\
	=&-\int_{F(\S)}\sum_{p\in F^{-1}(x)}\mathbf{H}_{\S}(p)\left<\xi(p),\nu(p)\right>\rd\mcH^2(x)\\
	\eqqcolon&-\int\left<X,\vec{\mathbf{H}}\right>\rd\mu.
}
On the other hand, it is not difficult to see that
\eq{
	A(t)
	=\mbM((f_t)_\sharp V),
}
and because $\mbM(W)$ is the oriented area,
\eq{
	\mcW'(0)
	=\frac{\rd}{\rd t}\mid_{t=0}\mbM((f_t)_\sharp W).
}
By virtue of the well-known first variation formula for varifolds \cite{Sim83}, we get
\eq{
	(\de V-\cos\theta\de W)[X]
	=\mcF'(0)
	=-\int\left<X,\vec{\mathbf{H}}\right>\rd\mu,
}
as desired.
\end{proof}

\begin{remark}
\normalfont
We point out that condition \eqref{condi-winding-number} is irredundant, because without this constraint, one could simply construct an example as Fig. \ref{Fig-1}, so that we could find some point $x\in F(\p\S)$ at which we no longer expect the upper-semi continuity of the tilde density (Proposition \ref{Prop-uppersemi-continuous}(1)) to be valid.

Precisely, given $\theta\in(\frac\pi2,\pi)$, let $F:\S\ra\overline{\mbR^3_+}$ be a $\theta$-capillary immersion with $F(\p\S)$ as in Fig. \ref{Fig-1}.
At the point $x$, because of the orientation of $F(\p\S)$, we have
\eq{
\lim_{r\ra0^+}\frac{\eta(B_r(x))}{\pi r^2}
=-\frac12,
}
and hence
\eq{
\tilde\Theta^2(\mu-\cos\theta \eta,x)
=1+\cos\theta<1,
}
implying that the function $\cdot\mapsto\tilde\Theta(\mu-\cos\theta\eta,\cdot)$ is not upper semi-continuous at $x$.
\end{remark}

\begin{figure}[H]
	\centering
	\includegraphics[width=12cm]{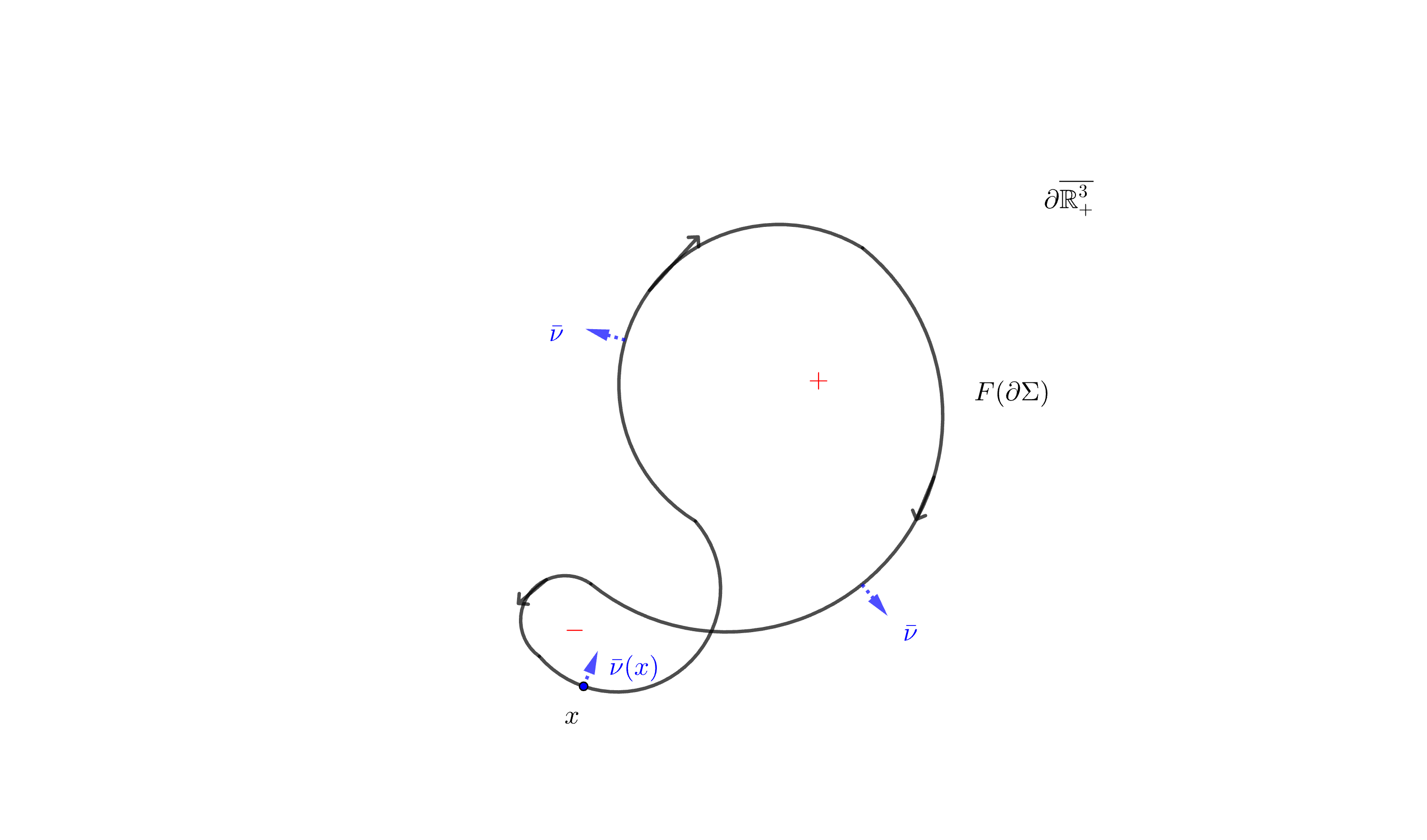}
	\caption{Example.}
	\label{Fig-1}
\end{figure}

}

\subsection{Willmore Energy in the Half-Space}\label{Sec-3-1}
The Willmore energy $\mcW(F)$ of a smooth immersed compact connected orientable surface $F:\S\ra\mbR^3$ with boundary $\p \S$, which may have several connected components, is usually proposed by 
\eq{\label{eq3.1}
	\mcW_0(F)
\coloneqq\frac14\int_{\S} \abs{{\bf H_\S}}^2\rd\mu_g+\int_{\p \S}\kappa_g,
}
where $\kappa_g$ denotes the geodesic curvature of $\p\S$ as a submanifold of $\S$.
The reason for proposing \eqref{eq3.1} is to keep the peculiarity of the Willmore functional—the invariance under the conformal transformation of the ambient space.
Thanks to the well-known Gauß-Bonnet theorem, the Willmore energy may be rewritten as
\eq{\label{eq3.2}\mcW_0(F)=\frac 14 \int_\Sigma\abs{\mathring{A}
}^2 +2\pi\chi (\S),
}
where $\mathring{A}$ is the traceless second fundamental form and $\chi(\S)$ is the Euler characteristic,
it is clear that the first term in \eqref{eq3.2} is conformal invariant, while the second term is topological invariant. See for example \cite{Schatzle10}.

The point in the above argument is that, in order to keep the conformal invariance, one may indeed add or subtract a topological quantity to the Willmore energy functional.
This insight motivates us to drop the term $\int_\S\kappa_g$ when considering the Willmore energy of the surface $\S$ that is capillary immersed into $\overline{\mbR^3_+}$, since this term is just a multiple of some topological quantity.
Indeed,
for any $\theta$-capillary immersion $F:\S\ra\overline{\mbR^3_+}$, it holds locally
\eq{
	\kappa_g
	=\na_\tau\mu_\S\cdot\tau
	=\na_\tau(\cos\theta\bar\nu-\sin\theta E_{3})\cdot\tau
	=\cos\theta\tilde\kappa_g,
}
where $\tilde\kappa_g$ is the geodesic curvature of $F(\p\S)$ as an {immersion} in $\p \mbR^3_+$,
using the Gauß-Bonnet theorem, one sees
{\color{black}\eq{
\int_{\p\S}\kappa_g
=
\cos\theta\int_{\p \S}\tilde\kappa_g
=2\pi\cos\theta \ {\rm ind}(F(\p\S)),
}
where ${\rm ind}(F(\p\S))$ is rotation index of immersed plane curve $F(\p\S)$, which is again a topological invariant.

Therefore, in this paper, we regard the Willmore functional for capillary hypersurfaces in $\mbR^{n+1}_+$ as 
\eq{\label{eq3.3}
	\mcW(F) &= \mcW_0(F)-2\pi\cos\theta \ {\rm ind}(F(\p\S))\\
&=\frac14\int_{\S} \abs{{\bf H}_\S}^2\rd\mu_g.
}
$W(F)$ is again a conformal invariant with respect to conformal diffeomorphisms of $\mbR^3$. }
It matches the Willmore functional for capillary hypersurfaces in the unit ball as well, which will be discussed in the next subsection.

We may rewrite the Willmore energy for the $\theta$-capillary immersion as:
\eq{
\mcW(F)
=\frac14\int_\S \abs{{\bf H}_\S}^2\rd\mu_g
=\frac14\int\abs{\vec {\bf H}}^2\rd\mu.
}


The first part of Theorem \ref{Thm-Li-Yau-Ineq0} amounts to be a direct application of  the monotonicity formula we obtained above. 


\begin{proof} [Proof of Theorem \ref{Thm-Li-Yau-Ineq0}, (1)-(3)]
Let us consider the point $x_0\in\p\S\subset\p\mbR^3_+$.
In this case, $\tilde x_0=x_0$ and hence $B_r(x_0)=\hat B_r(x_0)$. It follows that
\eq{ \lim_{r\searrow0^+}\frac{\eta(B_r(x_0))+ \eta( \hat B_r(x_0))}{\pi r^2}=N(x_0),
}
for some positive integer $N(x_0) =\mcH^0(F^{-1}(x_0))$,
since {\color{black}$F:\p\S\ra\p\mbR^3_+$ is an immersion}.
Thus we find
\eq{\label{eq3.0}
\tilde\Theta^2(\mu-\cos\theta\eta,x_0)
=2\Theta^2(\mu,x_0)-\cos\theta{\color{black}\cdot N(x_0).}
}
Moreover, since $\mu$ is supported on the upper half-space, $\Theta^2(\mu,x_0)=\frac{N(x_0)}{2}$.
Hence
\eq{\label{eq3.00}
\tilde\Theta^2(\mu-\cos\theta\eta,x_0)
=(1-\cos\theta) N(x_0).
}

By virtue of the above observation, we infer easily from the monotonicity identity \eqref{eq2.6} that
\eq{\label{eq-W(F)}
\mcW(F)
=4\int_{\mbR^3}\Abs{\frac14\vec {\bf H}+\frac{(x-x_0)^\perp}{\abs{x-x_0}^2}}^2\rd\mu+4\pi\Theta^2(\mu,x_0)-2\cos\theta\pi N(x_0).
}
Thus, we obtain the Li-Yau-type inequality:
\eq{
\mcW(F)
\geq 4(1-\cos\theta)\pi\Theta^2(\mu,x_0)=
2(1-\cos\theta)\pi N(x_0),
}
which implies easily that 
\eq{
\mcW(F)
\geq2(1-\cos\theta) \pi,
}
the sharp estimate on the lower bound of Willmore energy.

If $W(F)<4(1-\cos\theta)\pi$, for any $x_0\in \p\S$, the Li-Yau-type inequality \eqref{eq_ly} implies that $N(x_0)=1$, and hence $F:\p\S\ra\p\mbR^3_+$ must be an embedding.

Let us now consider the point $y_0\in\S\setminus\p\S$. 
Observe that
\eq{
\tilde\Theta^2(\mu-\cos\theta\eta,y_0)
=\Theta^2(\mu,y_0).
}
Letting $a=y_0$ in \eqref{iden-monotone-Theta} and recalling the definition of $\mcW(F)$, since $\theta\in[\frac\pi2,\pi)$, we obtain readily
\eq{
2\pi\Theta^2(\mu,y_0)
\leq\mcW(F),  
}
and hence $\Theta^2(\mu,y_0)=1$ if $\mcW(F)<4\pi$, in which case $F:\S\ra\overline{\mbR^3}_+$ must be an embedding. 
This completes the proof.
\end{proof}

\begin{remark}
\normalfont
When $F:\Sigma\ra\overline{\mbR^3_+}$ is an embedding, one can infer easily from \eqref{eq3.0} that
\eq{
\tilde\Theta^2(\mu-\cos\theta\eta,x_0) =1-\cos \theta,
} for $x_0\in \p\S$.
Therefore, it may be convenient to define
\eq{
\frac{\tilde\Theta^2(\mu-\cos\theta\eta,x_0)}{1-\cos\theta}
}
as the \textit{capillary density} of $\S$ at a boundary point $ x_0\in \p\S$.
\end{remark}

A further investigation gives the following characterization of the situation when $\mcW(F)$ attains its minima, which also reveals the sharpness of \eqref{eq_ly2}.
\begin{theorem}\label{Thm-Rigidity-Willmore-1}
Given $\theta\in(0,\pi)$, let $F:\S\ra\mbR^3_+$ be  a capillary immersion.
If $\mcW(F)=2\pi(1-\cos\theta)$,
then $F(\S)$ must be a $\theta$-cap in $\overline{\mbR^3_+}$, i.e., a part of the sphere interesting
$\p\mbR^{n+1}_+$ at the angle $\theta$.
\end{theorem}
\begin{proof}
We begin by noticing that $F:\p\S\ra\p\mbR^3_+$ is an embedding thanks to Theorem \ref{Thm-Li-Yau-Ineq0}.
In particular we learn from \eqref{eq-W(F)}
that
\eq{\label{eq-Volkmann16-(15)}
\frac14\vec {\bf H}(x)+\frac{(x-x_0)^\perp}{\abs{x-x_0}^2}=0
}
for $\mu$-a.e. $x\in F(\S)$ and any $x_0\in F(\p\S)$.
It is not difficult to observe that there must be some $y\in F(\S\setminus\p\S)$ such that $\vec {\bf H}(y)\neq0$, otherwise for $\mu$-a.e. $x\in F(\S)$, one has (up to a translation, we assume that $x_0=0$)
\eq{
x\cdot\nu(x)=0,
}
implying that the enclosing region of $F(\S)$ with $\p\mfR^3_+$ is a cone (see e.g. \cite{Mag12}*{Proposition 28.8}), which is not possible.
Consequently, by letting $x=y$ in \eqref{eq-Volkmann16-(15)}, we obtain for any $x_0\in F(\p\S)$ that
\eq{
\Abs{x_0-(y-\frac{2\nu (y)}{\abs{\vec {\bf H}(y)}})}^2
=\frac 4{\abs{\vec {\bf H}(y)}^2},
}
which shows that $F(\p\S)$ is indeed a $1$-dimensional sphere on $\p\mbR^3_+$, say $\mathfrak{s}$.

{\color{black}
To proceed, we consider the unique spherical cap in $\overline{\mbR^3_-}$, say $C_{\pi-\theta}$, which intersects $\p\mbR^3_+$ along $\mathfrak{s}$ with the constant angle $(\pi-\theta)$.
Let $V_{C_{\pi-\theta}}$ denote the naturally induced varifold of $C_{\pi-\theta}$, with weight measure denoted by $\mu_{C_{\pi-\theta}}$, which is exactly given by $\mu_{C_{\pi-\theta}}=\mcH^2\llcorner C_{\pi-\theta}$.
Write $\tilde\mu=\mu+\mu_{C_{\pi-\theta}}$, $\tilde\S={\rm spt}(\tilde\mu)=F(\S)\cup C_{\pi-\theta}$.

Since $F$ is a $\theta$-capillary immersion with $F:\p\S\ra\p\mbR^3_+$ an embedding, and because of the construction of $C_{\pi-\theta}$,  we see that the integral varifold $\tilde\mu$ has the following variational structure:
for any $X\in C_c^1(\mbR^3;\mbR^3)$,
\eq{
\int{\rm div}_{\tilde\S} X\rd\tilde\mu
=-\int_{\tilde\S}X\cdot\vec {\bf H}\rd\tilde\mu,
}
where $\vec {\bf H}$ is  the generalized mean curvature vector of $\tilde\S$, which is square integrable on $\tilde\S$.

Moreover, a direct computation shows that for the spherical cap $C_{\pi-\theta}$,
\eq{
\mcW(C_{\pi-\theta})
=\frac14\int_{C_{\pi-\theta}} \abs{\vec {\bf H}}^2\rd\mcH^2
=2(1-\cos (\pi-\theta)
)\pi=2(1+\cos \theta)\pi
,
}
which, together
with $\mcW(F)=2(1-\cos \theta)\pi$, yields
\eq{
\mcW(\tilde\mu)
=\frac14\int\abs{\vec {\bf H}}^2\rd\tilde\mu
=4\pi.
}
Here $\mcW(\tilde\mu)$ is the Willmore energy of $\tilde\mu$ in $\mbR^3$, see \cite{KS04}*{Appendix A}.
It is worth noting that a simple modification of \cite{Volkmann16}*{Proposition 4.3} in conjunction with \cite{KS12}*{Proposition 2.1.1} shows that for any compactly supported integral $2$-varifold with square integrable generalized mean curvature, if its Willmore energy is $4\pi$, then its support is exactly a closed round $2$-sphere.
Therefore we conclude that $\tilde\S$ must be a closed round $2$-sphere, and hence $F(\S)$ is a $\theta$-capillary spherical cap as desired.
}
\end{proof}
The last part of the assertion in Theorem \ref{Thm-Li-Yau-Ineq0}
is simply given by Theorem \ref{Thm-Rigidity-Willmore-1}.

\subsection{Willmore Energy in the Unit Ball}\label{sec-3-2}
Due to the conformal invariance of the Willmore functional, the results above can be transferred directly to hold for capillary surfaces in the unit ball.

It is well-known that $\mcW_0$, and hence $\mcW$ defined in \eqref{eq3.3}, are invariant under conformal transformations.
Therefore the Willmore functional for capillary surfaces in the unit ball is the same as $\mcW$,
i.e.,
{\color{black}\eq{\label{defn-W(F)-ball}
\mcW(F)
\coloneqq\frac14\int_{\S} \abs{\vec {\bf H}}^2\rd\mu_g+\int_{\p \S} \kappa_g-2\pi\cos\theta \ {\rm ind}(F(\p \S)) .
}
Here ${\rm ind}(F(\p \S))$ is the rotation index of the immersed plane curve  $\phi\circ F(\p \S)$ where $\phi: \mbS^3\setminus\{p\}\to\mbR^3$ is the stereographic projection for some $p\notin \p\S$.}

Now in this case, the geodesic curvature 
is given by
\eq{
	\kappa_g
	=\na_\tau\mu_\S\cdot\tau
	=\na_\tau(\cos\theta\bar\nu+\sin\theta \bar N(x))\cdot\tau
	=\cos\theta\tilde\kappa_g+\sin\theta,
}
where $\bar N(x)=x$ is the outer unit normal of $\mbB^3$ and $\tilde \kappa_g$ is the geodesic curvature of $F:\p\S\to \mbS^2$. {\color{black}By the Gauß-Bonnet theorem, we have
\eq{
2\pi \ {\rm ind}(F(\p \S))
=\int_{\p\S}\tilde\kappa_g +\abs{T},
}
where $\abs{T}$ is the oriented area of $T$.}
Altogether implying that 
\eq{\label{defn-Willmore-ball}
\mcW(F)
=\frac14\int_{\S} \abs{\vec {\bf H}}^2\rd\mu_g+
\sin\theta \abs{\p\Sigma} -\cos\theta\abs{T}.
}

\begin{corollary}\label{coro3.4}
Given $\theta\in(0,\pi)$, let $F:\S\ra\overline{\mbB^3}$ be a $\theta$-capillary immersion.
Then, there holds
\eq{
\mcW(F)
\geq2(1-\cos\theta)\pi.
}
Equality holds if and only if $F(\S)$ is a spherical cap or a  totally geodesic disk in $\overline{\mbB^3}$.
\end{corollary}
\begin{proof}
As explained above, the inequality holds thanks to Theorem \ref{Thm-Li-Yau-Ineq0} and it suffices to consider the characterization of equality case.
Observe that in the half-space case, equality is achieved if and only if the capillary surfaces are spherical.
Since umbilicity preserves under conformal transformations, we readily see that in the half-ball case, $F(\S)$ should be either a  spherical cap or a totally geodesic disk. 
\end{proof}
If $\theta=\frac\pi2$, Corollary \ref{coro3.4} was proved in \cite{Volkmann16}.
Now we show how it simply implies Theorem \ref{coro3.5}.
\begin{proof}[Proof of Theorem \ref{coro3.5}]
From Corollary \ref{coro3.4},we find
\eq{
\sin\theta\abs{\p\S}-\cos\theta\abs{T} \geq2(1-\cos\theta)\pi.
}
From ${\rm div} _\S x=2$ in $\S$ we have
\eq{\label{eq-S-pS}
2\abs{\Sigma} &=\int_{\p\S}\<x, \mu_\S\>
=\int_{\p\S}\<\cos\theta \bar \nu+\sin\theta \bar N(x),x\>\\
&=\sin \theta \abs{\p\S}.
}
\eqref{eq_coro3.5} follows clearly.
\end{proof}

\section{Monotonicity Identities in the Unit Ball}\label{Sec-4}
\subsection{Set-ups}
Let $\mbB^3\subset\mbR^3$ be the Euclidean unit ball centered at the origin, $\mbS^2\coloneqq\p\mbB^3$ denotes the corresponding unit sphere.

In this section we adopt the following notations:
let $\xi:\mbR^3\setminus\{0\}\ra\mbR^3$ denote the spherical inversion with respect to $\mbS^2$.
Fix $x_0\in\mbR^3$ and $r>0$, $B_r(x_0)$ denotes the open ball of radius $r$ centered at $x_0$, and
\eq{
\hat B_r(x_0)
=B_{\frac{r}{\abs{x_0}}}(\xi(x_0))
}
denotes the ball of radius $\frac{r}{\abs{x_0}}$ centered at $\xi(x_0)$.
A direct computation then shows that: for $x\in\mbS^2$,
\eq{\label{eq-x_0-xi(x_0)}
\abs{x_0}\abs{x-\xi(x_0)}
=\abs{x-x_0},
}
in other words, $B_r(x_0)\cap\mbS^2=\hat B_r(x_0)\cap\mbS^2$.

Similar with Definition \ref{Defn-contact-pair}, we consider the following weak formulation of capillary surfaces in the unit ball.
\begin{definition}\label{Defn-contact-pair-Ball}
\normalfont
Given $\theta\in(0,\pi)$, let $V$ be a rectifiable $2$-varifold supported on $\overline{\mbB^3}$ with its weight measure denoted by $\mu$, let $W$ be a rectifiable 2-varifold supported on $\mbS^2$ with its weight measure denoted by $\eta$.

$(V,W)$ is said to \textit{satisfy the contact angle condition $\theta$} if there exists a $\mu$-measurable vector field $\vec {\bf H}\in\mcL^1(\mbR^3,\mu)$ with $\vec {\bf H}(x)\in T_x\mbS^2$ for $\mu$-a.e. $x\in\mbS^2$, such that for every $X\in C_c^1(\mbR^3;\mbR^3)$ with $X$ tangent to $\mbS^2$, it holds that
\eq{\label{eq-1stvariation-capillary}
\int_{\mbR^3}{\rm div}_\S X\rd\mu-\cos\theta\int_{\mbS^2}{\rm div}_{\mbS^2}X\rd\eta
=-\int_{\mbB^3}\vec {\bf H}\cdot X\rd\mu,
}
where $\S\coloneqq{\rm spt}(\mu)$ is countably $2$-rectifiable.
\end{definition}
An important proposition for the pairs of varifolds satisfying contact angle condition is that they have
bounded first variation and satisfy the following first variation formula (see \cite{MP21}*{Proposition 3.1}).
\begin{proposition}\label{Prop-bounded-1st-variation}
Given $\theta\in[\frac{\pi}{2},\pi)$, let $(V,W)$ be as in Definition \ref{Defn-contact-pair}.
Then $V-\cos\theta W$ has bounded first variation.
More precisely, there exists a positive Radon measure $\sigma_V$ on $ \mbS^2$ such that
\eq{\label{eq-1stvariation-1}
\int_{\mbR^3}{\rm div}_\S X\rd\mu-\cos\theta\int_{\mbS^2}{\rm div}_{\mbS^2}X\rd\eta
=-\int_{\mbB^3}\vec {\bf H}\cdot X\rd\mu\\
+2\int_{\mbS^2} X(x)\cdot x\rd(\mu-\cos\theta\eta)+\int_{\mbS^2}X(x)\cdot x\rd\sigma_V
}
for every $X\in C_c^1(\mbR^3;\mbR^3)$.
\end{proposition}
If $\S$ is in fact a smooth capillary surface embedded in $\mbB^3$, then we may use a classical computation to see that the above formula is true with $\sigma_V=\sin\theta\mcH^2\llcorner\p\S$.
This motivates us to define $\gamma\coloneqq\frac{1}{\sin\theta}\sigma_V$ as the generalized boundary measure of $V$ and ${\rm spt}\gamma$ as the generalized boundary of $V$ (which are different from the ones obtained from Lebesgue's decomposition theorem).

In all follows, we consider only those $V$ that are integral $2$-varifolds with $\mu(\mbS^2)=0$ and $\vec {\bf H}\in\mcL^2(\mbR^3,\mu)$,
it follows that \eqref{eq-Volkmann16-(8)} holds.
Moreover, we may rewrite \eqref{eq-1stvariation-1} as
\eq{\label{eq-1stvariation-2}
\int_{\mbR^3}{\rm div}_\S X\rd\mu-\cos\theta\int_{\mbS^2}{\rm div}_{\mbS^2}X\rd\eta
=-\int_{\mbB^3}\vec {\bf H}\cdot X\rd\mu\\
-2\cos\theta\int_{\mbS^2}X(x)\cdot x\rd\eta+\sin\theta\int_{\mbS^2}X(x)\cdot x\rd\gamma
}
for every $X\in C_c^1(\mbR^3;\mbR^3)$.
An immediate consequence is that we may use a compactly supported vector field which coincides with the position vector field in a neighborhood of $\mbB^3$ to test \eqref{eq-1stvariation-2} and obtain
\eq{\label{eq-Volkmann16-(4)}
2\mu(\mbR^3)
=-\int_{\mbB^3}\vec {\bf H}\cdot x\rd\mu+\sin\theta\gamma(\mbS^2).
}
\subsection{Simon-Type Monotonicity Formulas}
Let us first recall that in \cite{Volkmann16}, a Simon-type monotonicity formula is proved for integral $2$-varifold $\mu$ with $\mu(\mbS^2)=0$, that has free boundary in $\overline{\mbB^3}$.
Precisely, it is proved that for $x_0\neq0$,
\eq{\label{eq-Volkmann16-(5)} &
\frac{1}{\pi}\int_{B_\rho(x_0)\setminus B_\sigma(x_0)}\Abs{\frac14\vec {\bf H}+\frac{(x-x_0)^\perp}{\abs{x-x_0}^2}}^2\rd\mu
+\frac{1}{\pi}\int_{\hat B_\rho(x_0)\setminus\hat B_\sigma(x_0)}\Abs{\frac14\vec {\bf H}+\frac{(x-\xi(x_0))^\perp}{\abs{x-\xi(x_0)}^2}}^2\rd\mu\\
=&\left(g_{x_0}(\rho)+\hat g_{x_0}(\rho)\right)-\left(g_{x_0}(\sigma)+\hat g_{x_0}(\sigma)\right),
}
where
\eq{
g_{x_0}(r)
\coloneqq
\frac{\mu(B_r(x_0))}{\pi r^2}+\frac{1}{16\pi}\int_{B_r(x_0)}\abs{\vec H}^2\rd\mu+\frac{1}{2\pi r^2}\int_{B_r(x_0)}\vec {\bf H}\cdot(x-x_0)\rd\mu,\\
\hat g_{x_0}(r)
\coloneqq
g_{\xi(x_0)}(\frac{r}{\abs{x_0}})-\frac{\abs{x_0}^2}{\pi r^2}\int_{\hat B_r(x_0)}\left(\abs{x-\xi(x_0)}^2+(x-\xi(x_0))^T\cdot x\right)\rd\mu\\
-\frac{\abs{x_0}^2}{2\pi r^2}\int_{\hat B_r(x_0)}\vec {\bf H}\cdot(\abs{x-\xi(x_0)}^2x)\rd\mu
+\frac{1}{2\pi}\int_{\hat B_r(x_0)}\vec {\bf H}\cdot x\rd\mu+\frac{\mu(\hat B_r(x_0))}{\pi},
}
and $(\cdot)^T$ denotes the orthogonal projection of a vector onto the approximate tangent space $T_x\S$.
For $x_0=0$, he obtained
\eq{\label{eq-Volkmann16-(5')}
\frac{1}{\pi}\int_{B_\rho(0)\setminus B_\sigma(0)}\Abs{\frac14\vec {\bf H}+\frac{1}{\pi}\frac{x^\perp}{\abs{x}^2}}^2\rd\mu
=\left(g_0(\rho)+\hat g_0(\rho)\right)-\left(g_0(\sigma)+\hat g_0(\sigma)\right),
}
where
\eq{
g_{0}(r)
\coloneqq&
\frac{\mu(B_r(0))}{\pi r^2}+\frac{1}{16\pi}\int_{B_r(0)}\abs{\vec {\bf H}}^2\rd\mu+\frac{1}{2\pi r^2}\int_{B_r(0)}\vec {\bf H}\cdot x\rd\mu,\\
\hat g_0(r)\coloneqq&-\frac{\min(r^{-2},1)}{2\pi}\left(2\mu(\mbR^3)+\int_{\mbR^3}\vec {\bf H}\cdot x\rd\mu\right).
}
Volkmann's proof is based on using a properly chosen vector field to test the first variation of a free boundary varifold. This particular choice can be dated back to \cite{Brendle12}.
We follow to use this vector field to test \eqref{eq-1stvariation-capillary} and obtain the following result.
\begin{theorem}\label{Thm-in-ball}
Given $\theta\in[\frac{\pi}{2},\pi)$, let $(V,W)$ be as in Definition \ref{Defn-contact-pair-Ball}.
Suppose in addition that $V$ is an integral $2$-varifold with $\mu(\mbS^2)=0$, then it holds that: for $x_0\neq0$, and for every $0<\sigma<\rho<\infty$,
\eq{\label{iden-Simon-capil-x0}
&\frac{1}{\pi}\int_{B_\rho(x_0)\setminus B_\sigma(x_0)}\Abs{\frac14\vec {\bf H}+\frac{(x-x_0)^\perp}{\abs{x-x_0}^2}}^2\rd\mu
+\frac{1}{\pi}\int_{\hat B_\rho(x_0)\setminus\hat B_\sigma(x_0)}\Abs{\frac14\vec {\bf H}+\frac{(x-\xi(x_0))^\perp}{\abs{x-\xi(x_0)}^2}}^2\rd\mu\\
&-\frac{\cos\theta}{\pi}\int_{B_\rho(x_0)\setminus B_\sigma(x_0)}\left(\frac{x-x_0}{\abs{x-x_0}^2}\cdot x\right)^2\rd\eta
-\frac{\cos\theta}{\pi}\int_{\hat B_\rho(x_0)\setminus \hat B_\sigma(x_0)}\left(\frac{x-\xi(x_0)}{\abs{x-\xi(x_0)}^2}\cdot x\right)^2\rd\eta\\
=&\left(g_{x_0,\theta}(\rho)+\hat g_{x_0,\theta}(\rho)\right)-\left(g_{x_0,\theta}(\sigma)+\hat g_{x_0,\theta}(\sigma)\right),
}
where
\eq{
g_{x_0,\theta}(r)
\coloneqq& g_{x_0}(r)-\frac{\cos\theta\eta(B_r(x_0))}{\pi r^2},\\
\hat g_{x_,\theta}(r)
\coloneqq&\hat g_{x_0}(r)-\frac{\cos\theta\abs{x_0}^2\eta(\hat B_r(x_0))}{\pi r^2}\\
&+\frac{\cos\theta\abs{x_0}^2}{\pi r^2}\int_{\hat B_r(x_0)}\abs{x-\xi(x_0)}^2\rd\eta-\frac{\cos\theta\eta(\hat B_r(x_0))}{\pi}.
}
And for $x_0=0$,
\eq{\label{iden-Simon-capil-0}
\frac{1}{\pi}\int_{B_\rho(0)\setminus B_\sigma(0)}\Abs{\frac14\vec {\bf H}+\frac{1}{\pi}\frac{x^\perp}{\abs{x}^2}}^2\rd\mu
=\left(g_{0}(\rho)+\hat g_{0,\theta}(\rho)\right)-\left(g_{0}(\sigma)+\hat g_{0,\theta}(\sigma)\right),
}
where
\eq{
\hat g_{0,\theta}(r)\coloneqq&
-\frac{\min(r^{-2},1)}{2\pi}\sin\theta\gamma(\mbS^2).
}
\end{theorem}
\begin{proof}
For every $0<\sigma<\rho<\infty$, we continue to use the Lipschitz cut-off function $l$ defined in the proof of Theorem \ref{Thm-Simon-halfspace}.
For $x_0\in\mbR^3$, let
\eq{
X_1(x)
\coloneqq l(\abs{x-x_0})(x-x_0).
}

\noindent{\bf Case 1. }$x_0\neq0$.

Define
\eq{
X_2(x)
\coloneqq &
\left((\frac{1}{\abs{x-\xi(x_0)}_{\frac{\sigma}{\abs{x_0}}}})^2-\frac{\abs{x_0}^2}{\rho^2}\right)_+(x-\xi(x_0))\\
&+\left(\frac{\min(\abs{x_0}\abs{x-\xi(x_0)},\rho)^2}{\rho^2}-\frac{\min(\abs{x_0}\abs{x-\xi(x_0)},\sigma)^2}{\sigma^2}\right)x,
}
and
\eq{
X\coloneqq X_1+X_2.
}
As verified in \cite{Volkmann16}*{(7)}, $X$ is an admissible vector field to test \eqref{eq-1stvariation-capillary} for a.e. $0<\sigma<\rho<\infty$.
Moreover, the terms $\int_{\mbR^3}{\rm div}_\S X\rd\mu$ and $-\int_{\mbR^3}\vec H\cdot X\rd\mu$ are explicitly computed. That is, if $\theta=\frac{\pi}{2}$, testing \eqref{eq-1stvariation-capillary} with such $X$, one gets exactly \eqref{eq-Volkmann16-(5)}.
Therefore for the case $\theta\in(\frac{\pi}{2},\pi)$, it suffices to compute
\eq{
-\cos\theta\int_{\mbS^2}{\rm div}_{\mbS^2}X\rd\eta.
}
To this end, we consider the following decomposition of $\mbR^3$:
\eq{
\mcF_1\coloneqq\{B_\sigma(x_0),B_\rho(x_0)\setminus B_\sigma(x_0),\mbR^3\setminus B_\rho(x_0)\},\\
\mcF_2\coloneqq\{\hat B_\sigma(x_0),\hat B_\rho(x_0)\setminus \hat B_\sigma(x_0),\mbR^3\setminus \hat B_\rho(x_0)\},
}
and we shall compute
\eq{
-\cos\theta\int_A{\rm div}_{\mbS^2}X_i\rd\eta
}
for all sets $A\in\mcF_i$, i=1,2, separately.

For $X_1$, a direct computation shows that
\eq{
\na X_1
=
\begin{cases}
\left(\frac{1}{\sigma^2}-\frac{1}{\rho^2}\right){\rm Id},\quad&0\leq\abs{x-x_0}\leq\sigma,\\
\left(\frac{1}{\abs{x-x_0}^2}-\frac{1}{\rho^2}\right){\rm Id}-2\frac{x-x_0}{\abs{x-x_0}^2}\otimes\frac{x-x_0}{\abs{x-x_0}^2},\quad&\sigma<\abs{x-x_0}\leq\rho\\
0,\quad&\rho<\abs{x-x_0},
\end{cases}
}
so that on $S^2$,
\eq{
{\rm div}_{\mbS^2}X_1(x)
=&{\rm div}X_1(x)-\na X_1[x]\cdot x\\
=&
\begin{cases}
\frac{2}{\sigma^2}-\frac{2}{\rho^2},\quad&0\leq\abs{x-x_0}\leq\sigma,\\
-\frac{2}{\rho^2}+2\left(\frac{x-x_0}{\abs{x-x_0}^2}\cdot x\right)^2,\quad&\sigma<\abs{x-x_0}\leq\rho,\\
0\quad&\rho<\abs{x-x_0},
\end{cases}
}
and it is easy to deduce that
\eq{\label{eq-div-S^2-X1}
&-\cos\theta\int_{\mbS^2}{\rm div}_{\mbS^2}X_1\rd\eta\\
=&-\cos\theta\left(\frac{2}{\sigma^2}\eta(B_\sigma(x_0))-\frac{2}{\rho^2}\eta(B_\rho(x_0))+2\int_{B_\rho(x_0)\setminus B_\sigma(x_0)}\left(\frac{x-x_0}{\abs{x-x_0}^2}\cdot x\right)^2\rd\eta\right).
}

For $X_2$, we may compute directly to find:

\noindent As $0\leq\abs{x-\xi(x_0)}\leq\frac{\sigma}{\abs{x_0}}$,
\eq{
X_2
=
\abs{x_0}^2\left(\frac{1}{\sigma^2}-\frac{1}{\rho^2}\right)(x-\xi(x_0))-\abs{x_0}^2\abs{x-\xi(x_0)}^2\left(\frac{1}{\sigma^2}-\frac{1}{\rho^2}\right)x;
}
as $\frac{\sigma}{\abs{x_0}}<\abs{x-\xi(x_0)}\leq\frac{\rho}{\abs{x_0}}$,
\eq{
X_2
=\left(\frac{1}{\abs{x-\xi(x_0)}^2}-\frac{\abs{x_0}^2}{\rho^2}\right)(x-\xi(x_0))-x+\frac{\abs{x_0}^2\abs{x-\xi(x_0)}^2}{\rho^2}x;
}
and $X_2\equiv0$ as $\frac{\rho}{\abs{x_0}}<\abs{x-\xi(x_0)}$.

Thus $\na X_2$ vanishes on $\mbR^3\setminus\hat B_\rho(x_0)$; as $0<\abs{x-\xi(x_0)}\leq\frac{\sigma}{\abs{x_0}}$,
\eq{
\na X_2
=
\left(\frac{1}{\sigma^2}-\frac{1}{\rho^2}\right)\abs{x_0}^2\left\{(1-\abs{x-\xi(x_0)}^2){\rm Id}-2x\otimes(x-\xi(x_0))\right\},
}
and as $\frac{\sigma}{\abs{x_0}}\leq\abs{x-\xi(x_0)}\leq\frac{\rho}{\abs{x_0}}$,
\eq{
\na X_2
=&\left(\frac{1}{\abs{x-\xi(x_0)}^2}-\frac{\abs{x_0}^2}{\rho^2}\right){\rm Id}-2\frac{x-\xi(x_0)}{\abs{x-\xi(x_0)}^2}\otimes\frac{x-\xi(x_0)}{\abs{x-\xi(x_0)}^2}\\
&-\left(1-\frac{\abs{x_0}^2\abs{x-\xi(x_0)}^2}{\rho^2}\right){\rm Id}
+2\frac{\abs{x_0}^2}{\rho^2}x\otimes(x-\xi(x_0).
}
Therefore it is not difficult to compute that: as $0\leq\abs{x-\xi(x_0)}\leq\frac{\sigma}{\abs{x_0}},$
\eq{
{\rm div}_{\mbS^2}X_2(x)
=
2\abs{x_0}^2\left(1-\abs{x-\xi(x_0)}^2\right)(\frac{1}{\sigma^2}-\frac{1}{\rho^2}),
}
and as $\frac{\sigma}{\abs{x_0}}<\abs{x-\xi(x_0)}<\frac{\rho}{\abs{x_0}},$
\eq{
{\rm div}_{\mbS^2}X_2(x)
=-\frac{2\abs{x_0}^2\left(1-\abs{x-\xi(x_0)}^2\right)}{\rho^2}-2+2\left(\frac{x-\xi(x_0)}{\abs{x-\xi(x_0)}^2}\cdot x\right)^2.
}
It follows that
\eq{\label{eq-div-S^2-X2}
&-\cos\theta\int_{\mbS^2}{\rm div}_{\mbS^2}X_2\rd\eta\\
=&-\cos\theta\bigg[\left(\frac{2\abs{x_0}^2}{\sigma^2}\eta(\hat B_\sigma(x_0))-\frac{2\abs{x_0}^2}{\sigma^2}\int_{\hat B_\sigma(x_0)}\abs{x-\xi(x_0)}^2\rd\eta
+2\eta(\hat B_\sigma(x_0))\right)\\
&-\left(\frac{2\abs{x_0}^2}{\rho^2}\eta(\hat B_\rho(x_0))-\frac{2\abs{x_0}^2}{\rho^2}\int_{\hat B_\rho(x_0)}\abs{x-\xi(x_0)}^2\rd\eta
+2\eta(\hat B_\rho(x_0))\right)\\
&+2\int_{\hat B_\rho(x_0)\setminus\hat B_\sigma(x_0)}\left(\frac{x-\xi(x_0)}{\abs{x-\xi(x_0)}^2}\cdot x\right)^2\rd\eta\bigg].
}
Recall that if $\theta=\frac{\pi}{2}$, testing \eqref{eq-1stvariation-capillary} with such $X$, one gets exactly \eqref{eq-Volkmann16-(5)}.
Now for $\theta\in(\frac{\pi}{2},\pi)$, taking \eqref{eq-div-S^2-X1} and \eqref{eq-div-S^2-X2} into consideration, we thus obtain \eqref{iden-Simon-capil-x0} for a.e. $\sigma$ and $\rho$, and an approximation argument shows that this indeed holds for every $\sigma$ and $\rho$.

\

\noindent{\bf Case 2. }$x_0=0$.

We continue to use $X_1$ defined in {\bf Case 1} (with $x_0=0$) and define
\eq{
X_2(x)
&\coloneqq\left(\frac{\min(1,\rho)^2}{\rho^2}-\frac{\min(1,\sigma)^2}{\sigma^2}\right)x,\\
X
&\coloneqq X_1+X_2.
}
A direct computation shows that $X$ is an admissible vector field to test \eqref{eq-1stvariation-capillary}.
Indeed, we have $X=X_1+X_2\equiv0$ on $\mbS^2$ for every $0<\sigma<\rho<\infty$.
Thus
\eq{
-\cos\theta\int_{\mbS^2}{\rm div}_{\mbS^2}X\rd\eta=0.
}
Recall that as $\theta=\frac{\pi}{2}$, testing \eqref{eq-1stvariation-capillary} with such $X$ one gets exactly \eqref{eq-Volkmann16-(5')}, therefore for $\theta\in(\frac{\pi}{2},\pi)$, we shall get the same identity.
Moreover,
thanks to \eqref{eq-Volkmann16-(4)} and invoking the definition of $\hat g_0$, we may rearrange this and deduce \eqref{iden-Simon-capil-0} as desired.
\end{proof}

\begin{proposition}
Given $\theta\in[\pi/2,\pi)$. For every $x_0\in\mbR^3$, the tilde-density
\eq{
&\tilde\Theta^2(\mu-\cos\theta\eta,x_0)\\
\coloneqq&
\begin{cases}
\lim_{r\searrow0}\left(\frac{(\mu-\cos\theta\eta(B_r(x_0))}{\pi r^2}+\frac{(\mu-\cos\theta\eta)(\hat B_r(x_0))}{\pi\abs{x_0}^{-2}r^2}\right),\quad&x_0\neq0,\\
\lim_{r\searrow0}\frac{\mu(B_r(0))}{\pi r^2},\quad&x_0=0,
\end{cases}
}
exists.
Moreover, the function $x\mapsto\tilde\Theta(\mu-\cos\theta\eta,x)$ is upper semi-continuous in $\mbR^3$.
\end{proposition}
\begin{proof}
We prove for the first assertion:

{\bf Case 1. }$x_0\neq0$.

Define
\eq{
R_{x_0}(r)
\coloneqq&\frac{1}{2\pi r^2}\int_{B_r(x_0)}\vec {\bf H}\cdot(x-x_0)\rd\mu+\frac{\abs{x_0}^2}{2\pi r^2}\int_{\hat B_r(x_0)}\vec {\bf H}\cdot(x-\xi(x_0))\rd\mu\\
&-\frac{\abs{x_0}^2}{\pi r^2}\int_{\hat B_r(x_0)}\left(\abs{x-\xi(x_0)}^2+(x-\xi(x_0))^T\cdot x\right)\rd\mu\\
&-\frac{\abs{x_0}^2}{2\pi r^2}\int_{\hat B_r(x_0)}\vec {\bf H}\cdot(\abs{x-\xi(x_0)}^2x)\rd\mu,\\
R_{x_0,\theta}(r)
\coloneqq& R_{x_0}(r)+\frac{\cos\theta\abs{x_0}^2}{\pi r^2}\int_{\hat B_r(x_0)}\abs{x-\xi(x_0)}^2\rd\eta-\frac{\cos\theta\eta(\hat B_r(x_0))}{\pi}\\
&+\frac{1}{2\pi}\int_{\hat B_r(x_0)}\vec {\bf H}\cdot x\rd\mu+\frac{\mu(\hat B_r(x_0))}{\pi},
}
then
\eq{
G_{x_0,\theta}(r)
\coloneqq&\frac{\left(\mu-\cos\theta\eta\right)(B_r(x_0))}{\pi r^2}+\frac{\abs{x_0}^2\left(\mu-\cos\theta\eta\right)(\hat B_r(x_0))}{\pi r^2}\\
&+\frac{1}{16\pi}\int_{B_r(x_0)}\abs{\vec {\bf H}}^2\rd\mu+\frac{1}{16\pi}\int_{\hat B_r(x_0)}\abs{\vec {\bf H}}^2\rd\mu+R_{x_0,\theta}(r)
}
is monotonically non-decreasing
thanks to \eqref{iden-Simon-capil-x0} and the fact that $\theta\in[\frac{\pi}{2},\pi)$.
Therefore
\eq{
\lim_{r\ra0^+}G_{x_0,\theta}(r)
}
exists.

Now we estimate with H\"older inequality:
\eq{\label{ineq-KS04-(A.5)-ball}
\abs{R_{x_0,\theta}(r)}
\leq&\left(\frac{\mu(B_r(x_0))}{\pi r^2}\right)^\frac12
\left(\frac{1}{4\pi}\int_{B_r(x_0)}\abs{\vec {\bf H}}^2\rd\mu\right)^\frac12\\
&+\left(\frac{\mu(\hat B_r(x_0))}{\pi \abs{x_0}^{-2}r^2}\right)^\frac12
\left(\frac{1}{4\pi}\int_{\hat B_r(x_0)}\abs{\vec {\bf H}}^2\rd\mu\right)^\frac12+\frac{\mu(\hat B_r(x_0))}{\pi}\\
&+\left(\frac{\mu(\hat B_r(x_0))}{\pi \abs{x_0}^{-2}r^2}\right)^\frac12\left(\frac{\mu(\hat B_r(x_0))}{\pi }\right)^\frac12+\left(\frac{\mu(\hat B_r(x_0))}{\pi }\right)^\frac12\left(\frac{1}{4\pi}\int_{\hat B_r(x_0)}\abs{\vec {\bf H}}^2\rd\mu\right)^\frac12\\
&-\frac{2\cos\theta\eta(\hat B_r(x_0))}{\pi}+\frac{1}{2\pi}\int_{\hat B_r(x_0)}\abs{\vec {\bf H}}\rd\mu+\frac{\mu(\hat B_r(x_0))}{\pi},
}
where we have used the fact that ${\rm spt}\mu\subset\overline{\mbB^3}$ so that $\abs{x}\leq1$.
Moreover, for $\frac18<\ep<\frac14$, Young's inequality gives
\eq{
\abs{R_{x_0,\theta}(r)}
\leq&\ep\frac{\mu(B_r(x_0))}{\pi r^2}+\frac{1}{16\pi\ep}\int_{B_r(x_0)}\abs{\vec {\bf H}}^2\rd\mu\\
&+\ep\frac{\mu(\hat B_r(x_0))}{\pi\abs{x_0}^{-2} r^2}+\frac{1}{16\pi\ep}\int_{\hat B_r(x_0)}\abs{\vec {\bf H}}^2\rd\mu\\
&+\ep\frac{\mu(\hat B_r(x_0))}{\pi\abs{x_0}^{-2}r^2}+\frac{1}{4\ep}\frac{\mu(\hat B_r(x_0))}{\pi}+\frac{\mu(\hat B_r(x_0))}{4\pi}+\frac{1}{4\pi}\int_{\hat B_r(x_0)}\abs{\vec {\bf H}}^2\rd\mu\\
&+2\frac{(\mu-\cos\theta\eta)(\hat B_r(x_0))}{\pi}+\frac{1}{2\pi}\int_{\hat B_r(x_0)}\abs{\vec {\bf H}}\rd\mu\\
\leq&2\ep\left(\frac{\mu(B_r(x_0))}{\pi r^2}+\frac{\mu(\hat B_r(x_0))}{\pi \abs{x_0}^{-2}r^2}\right)+\frac{\ep^{-1}}{16\pi}\left(\int_{B_r(x_0)\cup\hat B_r(x_0)}\abs{\vec {\bf H}}^2\rd\mu\right)\\
&+\frac{11+\ep^{-1}}{4\pi}(\mu-\cos\theta\eta)(\hat B_r(x_0))+\frac{3}{8\pi}\int_{\hat B_r(x_0)}\abs{\vec {\bf H}}^2\rd\mu,
}
where we have used again $\theta\in[\frac{\pi}{2},\pi)$ and the fact that
\eq{
\frac{1}{2\pi}\int_{\hat B_r(x_0)}\abs{\vec {\bf H}}\rd\mu
\leq\frac{\mu(\hat B_r(x_0))}{2\pi}+\frac{1}{8\pi}\int_{\hat B_r(x_0)}\abs{\vec {\bf H}}^2\rd\mu.
}

By virtue of the monotonicity of $G_{x_0,\theta}(r)$, we obtain for $0<\sigma<\rho<\infty$
\eq{&
\frac{\left(\mu-\cos\theta\eta\right)(B_\sigma(x_0))}{\pi\sigma^2}+\frac{\left(\mu-\cos\theta\eta\right)(\hat B_\sigma(x_0))}{\pi\abs{x_0}^{-2}\sigma^2}\\
\leq &\frac{\left(\mu-\cos\theta\eta\right)(B_\sigma(x_0))}{\pi\sigma^2}+\frac{\left(\mu-\cos\theta\eta\right)(\hat B_\sigma(x_0))}{\pi\abs{x_0}^{-2}\sigma^2}\\
&+\frac{1}{16\pi}\int_{B_\rho(x_0)}\abs{\vec {\bf H}}^2\rd\mu+\frac{1}{16\pi}\int_{\hat B_\rho(x_0)}\abs{\vec {\bf H}}^2\rd\mu+\abs{R_{x_0,\theta}(\rho)}+\abs{R_{x_0,\theta}(\sigma)}\\
\leq&(1+2\ep)\left(\frac{\left(\mu-\cos\theta\eta\right)(B_\rho(x_0))}{\pi\rho^2}+\frac{\left(\mu-\cos\theta\eta\right)(\hat B_\rho(x_0))}{\pi\abs{x_0}^{-2}\rho^2}\right)\\
&+\frac{1+2\ep^{-1}}{16\pi}\int_{B_\rho(x_0)\cup\hat B_r(x_0)}\abs{\vec {\bf H}}^2\rd\mu
+\frac{11+\ep^{-1}}{2\pi}(\mu-\cos\theta\eta)(\mbR^3)\\
&+\frac{3}{4\pi}\int_{\mbR^3}\abs{\vec H}^2\rd\mu
+2\ep\left(\frac{\left(\mu-\cos\theta\eta\right)(B_\sigma(x_0))}{\pi\sigma^2}+\frac{\left(\mu-\cos\theta\eta\right)(\hat B_\sigma(x_0))}{\pi\abs{x_0}^{-2}\sigma^2}\right).
}
In particular, since $\frac18\leq\ep\leq\frac14$, this yields for a fixed $R>0$ that
\eq{\label{eq-KS04-(A.6)-ball}
&\frac{\left(\mu-\cos\theta\eta\right)(B_r(x_0))}{\pi r^2}+\frac{\left(\mu-\cos\theta\eta\right)(\hat B_r(x_0))}{\pi\abs{x_0}^{-2}r^2}\\
\leq&3\left(\frac{\left(\mu-\cos\theta\eta\right)(B_R(x_0))}{\pi R^2}+\frac{\left(\mu-\cos\theta\eta\right)(\hat B_R(x_0))}{\pi\abs{x_0}^{-2}R^2}\right)\\
&+\frac{23}{4\pi}\int_{\mbR^3}\abs{\vec {\bf H}}^2\rd\mu+\frac{19}{\pi}(\mu-\cos\theta\eta)(\mbR^3)<\infty
}
for every $0<r<R$.

This, in conjunction with \eqref{ineq-KS04-(A.5)-ball}, yields
\eq{
\lim_{r\ra0^+}\abs{R_{x_0,\theta}(r)}=0,
}
and also implies that the tilde-density $\tilde\Theta(\mu-\cos\theta\eta,x_0)$ exists. Consequently
\eq{
\tilde\Theta(\mu-\cos\theta\eta,x_0)
\leq\left(\frac{\left(\mu-\cos\theta\eta\right)(B_r(x_0))}{\pi r^2}+\frac{\left(\mu-\cos\theta\eta\right)(\hat B_r(x_0))}{\pi\abs{x_0}^{-2}r^2}\right)\\
+\frac{1}{16\pi}\int_{B_r(x_0)}\abs{\vec {\bf H}}^2\rd\mu+\frac{1}{16\pi}\int_{\hat B_r(x_0)}\abs{\vec {\bf H}}^2\rd\mu+R_{\theta,x_0}(\rho).
}

{\bf Case 2. }$x_0=0$.

Define
\eq{
R_{0,\theta}(r)
\coloneqq\frac{1}{2\pi r^2}\int_{B_r(0)}\vec {\bf H}\cdot x\rd\mu,
}
then
\eq{
G_{0,\theta}(r)
\coloneqq\frac{\mu(B_r(0))}{\pi r^2}+\frac{1}{16\pi}\int_{B_r(0)}\abs{\vec {\bf H}}^2\rd\mu
-\frac{\min(r^{-2},1)}{2\pi}\sin\theta\gamma(\mbS^2)
+R_{0,\theta}(r)
}
is monotonically non-decreasing thanks to \eqref{iden-Simon-capil-0}, thus
\eq{
\lim_{r\ra0^+}G_{0,\theta}(r)
}
exists.

Using H\"older inequality and Young's inequality, we obtain for $\frac14\leq\ep\leq\frac12$,
\eq{\label{ineq-KS04-(A.5)-ball-0}
\abs{R_{0,\theta}(r)}
\leq&(\frac{\mu(B_r(0))}{\pi r^2})^\frac12(\frac{1}{4\pi}\int_{B_r(0)}\abs{\vec {\bf H}}^2\rd\mu)^\frac12\\
\leq&\ep\frac{\mu(B_r(0))}{\pi r^2}+\frac{1}{16\pi\ep}\int_{B_r(0)}\abs{\vec {\bf H}}^2\rd\mu.
}

By virtue of the monotonicity of $G_{0,\theta}(r)$, we obtain for $0<\sigma<\rho\leq1$
\eq{
\frac{\mu(B_\sigma(0))}{\pi\sigma^2}
\leq&\frac{\mu(B_\rho(0))}{\pi\rho^2}+\frac{1}{16\pi}\int_{B_\rho(0)}\abs{\vec {\bf H}}^2\rd\mu+\abs{R_{0,\theta}(\rho)}+\abs{R_{0,\theta}(\sigma)}\\
\leq&(1+\ep)\frac{\mu(B_\rho(0))}{\pi\rho^2}+\frac{1+2\ep^{-1}}{16\pi}\int_{B_\rho(0)}\abs{\vec {\bf H}}^2\rd\mu+\ep\frac{\mu(B_\sigma(0))}{\pi\sigma^2},
}
where we have used that $\hat g_{0,\theta}(r)\equiv-\frac{\sin\theta}{2\pi}\gamma(\mbS^2)$ for every $0<r\leq1$ in the first inequality.
In particular, since $\frac14\leq\ep\leq\frac12$ this yields for any $0<r<1$:
\eq{\label{eq-KS04-(A.6)-ball-0}
\frac{\mu(B_r(0))}{\pi r^2}
\leq3\frac{\mu(B_1(0))}{\pi}+\frac{9}{8\pi}\int_{\mbR^3}\abs{\vec {\bf H}}^2\rd\mu<\infty,
}
which, together with \eqref{ineq-KS04-(A.5)-ball-0}, yields
\eq{
\lim_{r\ra0^+}\abs{R_{0,\theta}(r)}=0,
}
and also implies that the tilde-density $\tilde\Theta(\mu-\cos\theta\eta,0)$ exists.
Consequently for $0<r<1$,
\eq{
\tilde\Theta(\mu,0)
\leq\frac{\mu(B_r(0))}{\pi r^2}+\frac{1}{16\pi}\int_{B_r(0)}\abs{\vec {\bf H}}^2\rd\mu+R_{0,\theta}(r).
}

Now we verify the upper semi-continuity by definition,
for a sequence of points $x_j\ra x_0$ and for $0<\rho<\frac12$, we have
\eq{
&\frac{(\mu-\cos\theta)(\overline{B_\rho(x_0)})}{\pi\rho^2}+\frac{(\mu-\cos\theta)(\overline{\hat B_\rho(x_0)})}{\pi\abs{x_0}^{-2}\rho^2}\\
\geq&\limsup_{j\ra\infty}\left(\frac{(\mu-\cos\theta)(\overline{B_\rho(x_j)})}{\pi\rho^2}+\frac{(\mu-\cos\theta)(\overline{\hat B_\rho(x_j)})}{\pi\abs{x_j}^{-2}\rho^2}\right)\\
\geq&\limsup_{j\ra\infty}\left(\tilde\Theta(\mu-\cos\theta\eta,x_j)-\frac{1}{16\pi}\left(\int_{B_\rho(x_j)}\abs{\vec {\bf H}}^2\rd\mu+\int_{\hat B_\rho(x_j)}\abs{\vec {\bf H}}^2\rd\mu\right)-R_{x_j,\theta}(\rho)\right)\\
\geq&\limsup_{j\ra\infty}\tilde\Theta(\mu-\cos\theta\eta,x_j)-C\bigg(\frac{(\mu-\cos\theta\eta)(B_\frac12(x_0)\cup\hat B_\frac12(x_0))}{\pi (\frac12)^2}\\
&+\int_{\mbR^3}\abs{\vec {\bf H}}^2\rd\mu
+(\mu-\cos\theta\eta)(\mbR^3)\bigg)
\cdot\bigg(\norm{\vec {\bf H}}_{L^2(B_{2\rho}(x_0))}+(\frac{\mu(\hat B_\rho(x_0))}{\pi})^\frac12\bigg)\\
&-2\frac{(\mu-\cos\theta\eta)(\hat B_\rho(x_0))}{\pi},
}
where we have used \eqref{eq-KS04-(A.6)-ball}, \eqref{eq-KS04-(A.6)-ball-0} for the last inequality, here we interpret $\hat B_r(0)=\emptyset$ and $\frac{(\mu-\cos\theta)(\hat B_\rho(0))}{\pi\abs{0}^{-2}\rho^2}=0$.
Letting $\rho\ra0^+$, this gives
\eq{
\tilde\Theta(\mu-\cos\theta\eta,x_0)\geq\limsup_{j\ra\infty}\tilde\Theta(\mu-\cos\theta\eta,x_j),
}
which completes the proof of the second assertion.
\end{proof}
\begin{remark}
\normalfont
Since ${\rm spt}\mu$ and ${\rm spt}\eta$ are compact, we see from the definition that
\eq{
&\lim_{r\ra\infty}R_{x_0,\theta}(r)=\frac{1}{2\pi}\int_{\mbR^3}\vec {\bf H}\cdot x\rd\mu+\frac{(\mu-\cos\theta\eta)(\mbR^3)}{\pi},\quad x_0\neq0,\\
&\lim_{r\ra\infty}R_{0,\theta}(r)=0,
}
and obtain from the Simon-type monotonicity formulas (letting $\sigma\ra0^+,\rho\ra\infty$)
\eq{\label{iden-Simon-ball-1} &
\frac{1}{\pi}\int_{\mbR^3}\abs{\frac14\vec {\bf H}+\frac{(x-x_0)^\perp}{\abs{x-x_0}^2}}^2+\abs{\frac14\vec {\bf H}+\frac{(x-\xi(x_0))^\perp}{\abs{x-\xi(x_0)}^2}}^2\rd\mu\\
&-\frac{\cos\theta}{\pi}\int_{\mbR^3}(\frac{x-x_0}{\abs{x-x_0}^2}\cdot x)^2+(\frac{x-\xi(x_0)}{\abs{x-\xi(x_0)}^2}\cdot x)^2\rd\eta\\
=&\lim_{\rho\ra\infty}G_{x_0,\theta}(\rho)-\lim_{\sigma\ra0^+}G_{x_0,\theta}(\sigma)\\
=&\frac{1}{8\pi}\int_{\mbR^3}\abs{\vec {\bf H}}^2\rd\mu+\frac{1}{2\pi}\int_{\mbR^3}\vec H\cdot x\rd\mu+\frac{(\mu-\cos\theta\eta)(\mbR^3)}{\pi}-\tilde\Theta(\mu-\cos\theta\eta,x_0)\\
=&\frac{1}{8\pi}\int_{\mbR^3}\abs{\vec {\bf H}}^2\rd\mu+\frac{\sin\theta\gamma(\mbS^2)-2\cos\theta\eta(\mbR^3)}{2\pi}-\tilde\Theta(\mu-\cos\theta\eta,x_0)
}
for $x_0\neq0$, and for $x_0=0:$
\eq{\label{iden-Simon-ball-1'}
\frac{1}{\pi}\int_{\mbR^3}\abs{\frac14\vec {\bf H}+\frac{x^\perp}{\abs{x}^2}}^2\rd\mu
=\frac{1}{16\pi}\int_{\mbR^3}\abs{\vec {\bf H}}^2\rd\mu+\frac{\sin\theta\gamma(\mbS^2)}{2\pi}-\tilde\Theta(\mu,0).
}
\end{remark}

\subsection{Applications of Eqn. \ref{iden-Simon-ball-1}}

As a by-product of the monotonicity formula, we may establish the lower bound for the Willmore functional, which is conformal invariant and was proved in the previous Section. Moreover, we may obtain the optimal area estimate for minimal capillary surfaces in the unit ball in a rather direct way, which needs not go through the
Willmore functional.

\

\begin{proof}[Another proof of Theorem \ref{coro3.5}]
In the case that $x_0\in\mbS^2$, we clearly have $\xi(x_0)=x_0$.
Besides, for any $x\in\mbS^2$, there holds $\abs{x-x_0}^4=4(1-(x\cdot x_0))^2$, and hence a direct computation shows that
\eq{
(\frac{x-x_0}{\abs{x-x_0}^2}\cdot x)^2+(\frac{x-\xi(x_0)}{\abs{x-\xi(x_0)}^2}\cdot x)^2
=\frac12\text{ on }\mbS^2.
}
Taking this into account, the Simon-type monotonicity formula \eqref{iden-Simon-ball-1} then reads
\eq{\label{iden-Simon-ball-2}
\frac{1}{\pi}\int_{\mbR^3}\abs{\frac14\vec {\bf H}+\frac{(x-x_0)^\perp}{\abs{x-x_0}^2}}^2+\abs{\frac14\vec {\bf H}+\frac{(x-\xi(x_0))^\perp}{\abs{x-\xi(x_0)}^2}}^2\rd\mu\\
=\frac{1}{8\pi}\int_{\mbR^3}\abs{\vec {\bf H}}^2\rd\mu+\frac{\sin\theta\gamma(\mbS^2)-\cos\theta\eta(\mbR^3)}{2\pi}-\tilde\Theta(\mu-\cos\theta\eta,x_0).
}
Now since $F$: $\Sigma \to  \mbR^3$ be a capillary minimal immersion, and also recall \eqref{eq-Volkmann16-(4)}, we obtain 
\eq{\label{iden-Simon-ball-3}
\frac{1}{\pi}\int_{\mbR^3}\abs{\frac{(x-x_0)^\perp}{\abs{x-x_0}^2}}^2+\abs{\frac{(x-\xi(x_0))^\perp}{\abs{x-\xi(x_0)}^2}}^2\rd\mu
=\frac{2\abs{\S}-\cos\theta\abs{T}}{2\pi}-\tilde\Theta(\mu-\cos\theta\eta,x_0).
}
As \eqref{eq3.00} we have $\tilde\Theta(\mu-\cos\theta\eta,x_0)=(1-\cos \theta) N(x_0)$, where $N(x_0) = \mcH^0(F^{-1}(x_0))$.
Theorem \ref{coro3.5} then follows.    
\end{proof}


\bibliographystyle{alpha}
\bibliography{BibTemplate.bib}

\end{document}